\documentclass[11pt,reqno]{amsart}
\usepackage{fullpage}
\usepackage{mathrsfs,amssymb,graphicx,verbatim,amsmath,amsfonts}
\usepackage{paralist}
\usepackage[breaklinks,pdfstartview=FitH]{hyperref}
\usepackage{upgreek}
\usepackage{tikz}
\usepackage[mathscr]{euscript}

\hypersetup{ocgcolorlinks=true,allcolors=testc}
\hypersetup{
     colorlinks   = true,
     citecolor    = black
}
\hypersetup{linkcolor=black}
\hypersetup{urlcolor=black}

\usepackage{dutchcal}

\renewcommand{\leq}{\leqslant}
\renewcommand{\geq}{\geqslant}
\renewcommand{\setminus}{\smallsetminus}
\renewcommand{\gamma}{\upgamma}
\allowdisplaybreaks

\usepackage{esint}
\usepackage[autostyle]{csquotes}

\renewcommand{\pi}{\uppi}

\newcommand{\e}{\varepsilon}
\newcommand{\R}{\mathbb R}

\newtheorem{theorem}{Theorem}
\newtheorem{lemma}[theorem]{Lemma}
\newtheorem{proposition}[theorem]{Proposition}

\newtheorem{corollary}[theorem]{Corollary}

\theoremstyle{remark}
\newtheorem{remark}[theorem]{Remark}

\renewcommand{\tau}{\uptau}

\renewcommand{\xi}{\upxi}
\renewcommand{\rho}{\uprho}

\renewcommand{\subset}{\subseteq}
\newcommand{\C}{\mathbb C}

\newcommand{\N}{\mathbb N}

\newcommand{\eqdef}{\stackrel{\mathrm{def}}{=}}

\renewcommand{\theta}{\uptheta}
\renewcommand{\lambda}{\uplambda}

\renewcommand{\gamma}{\upgamma}
\renewcommand{\beta}{\upbeta}
\renewcommand{\alpha}{\upalpha}
\renewcommand{\kappa}{\upkappa}
\renewcommand{\psi}{\uppsi}
\renewcommand{\rho}{\uprho}
\renewcommand{\delta}{\updelta}
\renewcommand{\pi}{\uppi}
\renewcommand{\omega}{\upomega}
\renewcommand{\sigma}{\upsigma}

\renewcommand{\eta}{\upeta}

\renewcommand{\kappa}{\upkappa}
\renewcommand{\mu}{\upmu}
\renewcommand{\nu}{\upnu}
\renewcommand{\pi}{\uppi}
\renewcommand{\zeta}{\upzeta}

\newcommand{\mb}{\mathbb}
\newcommand*\diff{\mathop{}\!\mathrm{d}}

\newcommand{\ms}{\mathscr}
\newcommand{\msf}{\mathsf}

\begin{document}

\title{Polynomial inequalities on the Hamming cube}

\author{Alexandros Eskenazis}
\address{(A.~E.) Department of Mathematics\\ Princeton University\\ Princeton, NJ 08544-1000, USA}
\email{ae3@math.princeton.edu}

\author{Paata Ivanisvili}
\address{(P.~I.) Department of Mathematics\\ University of California, Irvine\\ CA 92617, USA}
\email{pivanisv@uci.edu}

\thanks{P.~I.~ was partially supported by NSF DMS-1856486 and NSF CAREER-1945102. This work was carried out under the auspices of the Simons Algorithms and Geometry (A\&G) Think Tank.}

\vspace{-0.25in}

\begin{abstract} 
Let $(X,\|\cdot\|_X)$ be a Banach space. The purpose of this article is to systematically investigate dimension independent properties of vector valued functions $f:\{-1,1\}^n\to X$ on the Hamming cube whose spectrum is bounded above or below. Our proofs exploit contractivity properties of the heat flow, induced by the geometry of the target space $(X,\|\cdot\|_X)$, combined with duality arguments and suitable tools from approximation theory and complex analysis. We obtain a series of improvements of various well-studied estimates for functions with bounded spectrum, including moment comparison results for low degree Walsh polynomials and Bernstein--Markov type inequalities, which constitute discrete vector valued analogues of Freud's inequality in Gauss space (1971). Many of these inequalities are new even for scalar valued functions. Furthermore, we provide a short proof of Mendel and Naor's heat smoothing theorem (2014) for functions in tail spaces with values in spaces of nontrivial type and we also prove a dual lower bound on the decay of the heat semigroup acting on functions with spectrum bounded from above. Finally, we improve the reverse Bernstein--Markov inequalities of Meyer (1984) and Mendel and Naor (2014) for functions with narrow enough spectrum and improve the bounds of Filmus, Hatami, Keller and Lifshitz (2016) on the $\ell_p$ sums of influences of bounded functions for $p\in\big(1,\frac{4}{3}\big)$.
\end{abstract}

\maketitle

{\footnotesize

}

\section{Introduction}

Fix $n\in\N$ and let $(X,\|\cdot\|_X)$ be a Banach space. If $p\in[1,\infty)$, the vector valued $L_p$ norm of a function $f:\{-1,1\}^n\to X$ is defined as
\begin{equation}
\|f\|_{L_p(\{-1,1\}^n;X)} \eqdef \Big( \frac{1}{2^n} \sum_{\e\in\{-1,1\}^n} \|f(\e)\|_X^p\Big)^{1/p}.
\end{equation}
As usual, we denote $\|f\|_{L_\infty(\{-1,1\}^n;X)} \eqdef \max_{\e\in\{-1,1\}^n} \|f(\e)\|_X$.  For a subset $A\subseteq\{1,\ldots,n\}$ the Walsh function $w_A:\{-1,1\}^n \to \{-1,1\}$ is the Boolean function given by $w_A(\e) = \prod_{i\in A}\e_i$, where $\e=(\e_1,\ldots,\e_n)\in\{-1,1\}^n$. Every function $f:\{-1,1\}^n \to X$ admits an expansion of the form
\begin{equation}
f = \sum_{A\subseteq\{1,\ldots,n\}} \widehat{f}(A) w_A,
\end{equation}
where
\begin{equation}
\widehat{f}(A) \eqdef \frac{1}{2^n} \sum_{\delta\in\{-1,1\}^n} f(\delta) w_A(\delta) \in X.
\end{equation}
For $i\in\{1,\ldots,n\}$ the $i$-th partial derivative of such a function $f$ is given by
\begin{equation}
\partial_i f(\e) \eqdef \frac{f(\e)-f(\e_1,\ldots,\e_{i-1},-\e_i,\e_{i+1},\ldots,\e_n)}{2} = \sum_{\substack{A\subseteq\{1,\ldots,n\} \\ i\in A}} \widehat{f}(A) w_A(\e)
\end{equation}
and satisfies $\partial_i^2 f = \partial_i f$. Therefore, the hypercube Laplacian of $f$ is defined as
\begin{equation}
\Delta f \eqdef \sum_{i=1}^n \partial_i f = \sum_{A\subseteq\{1,\ldots,n\}} |A|\widehat{f}(A) w_A.
\end{equation}
Finally, the action of the discrete heat semigroup $\{e^{-t\Delta}\}_{t\geq0}$ on the function $f$ is given by
\begin{equation}
\forall \ t\geq0, \ \ \ e^{-t\Delta} f \eqdef \sum_{A\subseteq\{1,\ldots,n\}} e^{-t|A|}\widehat{f}(A) w_A.
\end{equation}
A straightforward calculation shows that the heat semigroup can equivalently be expressed as
\begin{equation} \label{eq:heatrepresent}
e^{-t\Delta}f(\e) =\frac{1}{2^n} \sum_{\delta\in\{-1,1\}^n} \sum_{B\subseteq\{1\ldots,n\}} e^{-t|B|}(1-e^{-t})^{n-|B|} f\Big(\sum_{j\in B} \e_j e_j + \sum_{j\in \{1,\ldots,n\}\setminus B} \delta_j e_j\Big),
\end{equation}
where $\{e_1,\ldots,e_n\}$ is the orthonormal basis of $\R^n$, which by convexity implies that for every $t\geq0$ and $p\in[1,\infty]$, $\|e^{-t\Delta}f\|_{L_p(\{-1,1\}^n;X)} \leq \|f\|_{L_p(\{-1,1\}^n;X)}$. Identity \eqref{eq:heatrepresent} also has a useful probabilistic interpretation. For a fixed $\e\in\{-1,1\}^n$ and $t\geq0$, consider a random vector $\eta\in\{-1,1\}^n$ such that each coordinate $\eta_i$ is chosen independently to coincide with $\e_i$ with probability $\frac{1+e^{-t}}{2}$ and with $-\e_i$ with probability $\frac{1-e^{-t}}{2}$. Then, the value $e^{-t\Delta}f(\e)$ is the expectation of the random vector $f(\eta)$.

The main purpose of this paper is to investigate finer contractivity properties of the heat semigroup under suitable assumptions on the spectrum of the function $f$ and the geometry of the target space $(X,\|\cdot\|_X)$. The common feature in the proofs of most of the following results is a duality argument inspired by classical work of Figiel (see \cite[Theorem~14.6]{MS86}), which allows the self-improvement of contractivity properties of the heat semigroup relying on suitable inequalities from classical approximation theory and complex analysis. All Banach spaces in the ensuing discussion will be assumed to be over the field of complex numbers.

In the rest of the introduction, we proceed to describe our results in decreasing order of generality. In Section \ref{subsec:1.1}, we present estimates for functions with spectrum bounded from above and values in a general Banach space. In Section \ref{subsec:1.2}, we improve the bounds of Section \ref{subsec:1.1} under the additional assumption that the target space $(X,\|\cdot\|_X)$ is $K$-convex (we postpone the relevant definitions until Section 1.2). We also present a new proof of a theorem of Mendel and Naor \cite{MN14} related to the heat smoothing conjecture. Finally, in Section \ref{subsec:1.3} we present explicit estimates which hold true for scalar valued functions. These include various Bernstein--Markov type inequalities and their reverses, estimates on the influences of bounded functions and moment comparison results.


\subsection{Estimates for a general Banach space} \label{subsec:1.1} Fix $n\in\N$ and let $(X,\|\cdot\|_X)$ be a Banach space. If $d\in\{0,1\ldots,n\}$, we say that a function $f:\{-1,1\}^n\to X$ has degree at most $d$ if $\widehat{f}(A)=0$ when $|A|>d$ and we say that $f$ belongs in the $d$-th tail space if $\widehat{f}(A)=0$ when $|A|<d$. Finally, we say that $f$ is $d$-homogeneous if $\widehat{f}(A)=0$ when $|A|\neq d$.

\subsubsection{A lower bound on the decay of the heat semigroup} The following theorem establishes a lower bound on the decay of the heat semigroup acting on functions of low degree.

\begin{theorem} \label{thm:reverseheatgeneral}
Fix $n,d\in\N$ with $d\in\{1,\ldots,n\}$ and let $(X,\|\cdot\|_X)$ be a Banach space. For every $p\in[1,\infty]$ and every function $f:\{-1,1\}^n \to X$ of degree at most $d$, we have
\begin{equation} \label{eq:reverseheatgeneral}
\forall \ t\geq0, \ \ \ \|e^{-t\Delta}f\|_{L_p(\{-1,1\}^n;X)} \geq \frac{1}{T_d(e^t)} \|f\|_{L_p(\{-1,1\}^n;X)},
\end{equation}
where $T_d$ is the $d$-th Chebyshev polynomial of the first kind.
\end{theorem}

We note that a weaker bound in the spirit of Theorem \ref{thm:reverseheatgeneral}, attributed partially to Oleszkiewicz, was established in \cite[Lemma~5.4]{FHKL16} (see Remark \ref{rem:ole} below for a comparison \mbox{with Theorem \ref{thm:reverseheatgeneral}).}

Using Theorem \ref{thm:reverseheatgeneral} and the hypercontractivity of the discrete heat semigroup (see \cite{Bon70}), we deduce the following moment comparison for functions of low degree.

\begin{corollary} \label{cor:realchaos}
Fix $n,d\in\N$ with $d\in\{1,\ldots,n\}$ and let $(X,\|\cdot\|_X)$ be a Banach space. For every $p>q>1$ and every function $f:\{-1,1\}^n\to X$ of degree at most $d$, we have
\begin{equation} \label{eq:realchaos}
\|f\|_{L_p(\{-1,1\}^n;X)} \leq T_d\Big(\sqrt{\frac{p-1}{q-1}}\Big) \|f\|_{L_q(\{-1,1\}^n;X)}.
\end{equation}
\end{corollary}

To the extent of our knowledge, the best previously known moment comparison for general vector valued functions of low degree on the discrete cube can be extracted from an argument in the monograph \cite{KW92} of Kwapie\'{n} and Woyczy\'{n}ski. In Remark \ref{rem:kwapien} below, we quantify their argument and show that the bounds of \cite[Proposition~6.5.1]{KW92} \mbox{are weaker than those of Corollary \ref{cor:realchaos}.}

Another application of Theorem \ref{thm:reverseheatgeneral} is a refinement of a celebrated inequality on the discrete cube due to Pisier \cite{Pis86}. In connection with his work on nonlinear type, Pisier showed that for every Banach space $(X,\|\cdot\|_X)$, every $p\in[1,\infty]$ and every function $f:\{-1,1\}^n\to X$, we have
\begin{equation} \label{eq:pisier}
\Big\| f - \frac{1}{2^n} \sum_{\delta\in\{-1,1\}^n} f(\delta) \Big\|_{L_p(\{-1,1\}^n;X)} \leq (\log n+1) \Big( \frac{1}{2^n} \sum_{\delta\in\{-1,1\}^n} \Big\| \sum_{i=1}^n \delta_i \partial_i f\Big\|^p_{L_p(\{-1,1\}^n;X)}\Big)^{1/p}.
\end{equation}
For $X=\R$, the right hand side of \eqref{eq:pisier} is equivalent to the $L_p$ norm of the gradient of $f$ due to Khintchine's inequality \cite{Khi23}, thus \eqref{eq:pisier} can be understood as a vector valued Poincar\'e inequality. The dependence on the dimension $n$ in Pisier's inequality \eqref{eq:pisier} for various classes of spaces $X$ is fundamental in investigations in the nonlinear geometry of Banach spaces and the Ribe program (see \cite{Nao12} for a detailed discussion around this topic). For general Banach spaces, Talagrand \cite{Tal93} has proven that the factor $\log n$ in \eqref{eq:pisier} is asymptotically optimal for every $p\in[1,\infty)$, whereas Wagner \cite{Wag00} has shown that when $p=\infty$, Pisier's inequality holds with an absolute constant. Here we show the following refinement of Pisier's inequality for functions of low degree.

\begin{theorem} \label{thm:Pisierlowfreq}
Fix $n,d\in\N$ with $d\in\{1,\ldots,n\}$ and let $(X,\|\cdot\|_X)$ be a Banach space. For every $p\in[1,\infty)$ and every function $f:\{-1,1\}^n \to X$ of degree at most $d$, we have
\begin{equation} \label{eq:Pisierlowfreq}
\Big\| f - \frac{1}{2^n} \sum_{\delta\in\{-1,1\}^n} f(\delta) \Big\|_{L_p(\{-1,1\}^n;X)} \leq 3(\log d+1) \Big( \frac{1}{2^n} \sum_{\delta \in \{-1,1\}^n} \Big\| \sum_{i=1}^n \delta_i \partial_i f \Big\|_{L_p(\{-1,1\}^n;X)}^p\Big)^{1/p}
\end{equation}
and the $\log d$ factor is asymptotically sharp for every $p\in[1,\infty)$.
\end{theorem}

\subsubsection{A Bernstein--Markov type inequality for $\Delta$} A variant of the proof of Theorem \ref{thm:reverseheatgeneral} implies the following Bernstein--Markov \mbox{type inequality, which had previously appeared in \cite{FHKL16}.}

\begin{theorem} \label{thm:laplaciangeneral}
Fix $n,d\in\N$ with $d\in\{1,\ldots,n\}$ and let $(X,\|\cdot\|_X)$ be a Banach space. For every $p\in[1,\infty]$ and every function $f:\{-1,1\}^n \to X$ of degree at most $d$, we have
\begin{equation} \label{eq:laplaciangeneral}
\| \Delta f\|_{L_p(\{-1,1\}^n;X)} \leq d^2 \|f\|_{L_p(\{-1,1\}^n;X)}.
\end{equation}
\end{theorem}

In \cite[Lemma~5.4]{FHKL16}, the Bernstein--Markov type inequality \eqref{eq:laplaciangeneral} was stated only for real valued functions and $p=1$. In the vector valued setting which is of interest here, inequality \eqref{eq:laplaciangeneral} is sharp for general Banach spaces. Recall that a Banach space $(X,\|\cdot\|_X)$ has cotype $q\in[2,\infty]$ with constant $C\in(0,\infty)$ if for every $n\in\N$ and vectors $x_1,\ldots,x_n\in X$, we have
\begin{equation}
\Big(\frac{1}{2^n} \sum_{\e\in\{-1,1\}^n} \Big\|\sum_{i=1}^n \e_i x_i\Big\|_X^q\Big)^{1/q}  \geq \frac{1}{C} \Big(\sum_{i=1}^n \|x_i\|_X^q\Big)^{1/q}.
\end{equation}
The fact that every Banach space $(X,\|\cdot\|_X)$ has cotype $q=\infty$ with constant $C=1$ follows from the triangle inequality, yet having finite cotype is a meaningful structural property of a given space. We refer to the survey \cite{Mau03} for further information on the rich theory of type and cotype. We will show that any improvement of \eqref{eq:laplaciangeneral}\mbox{, forces the target space $X$ to have finite cotype.}

\begin{theorem} \label{thm:cotype}
Let $(X,\|\cdot\|_X)$ be a Banach space and $p\in[1,\infty]$. Suppose that there exist some $\eta\in(0,1)$ and $d\in\N$ such that for every $n\in\N$ with $d\in\{1,\ldots,n\}$, every function $f:\{-1,1\}^n \to X$ of degree at most $d$ satisfies the inequality
\begin{equation} \label{eq:cotypeassumption}
\|\Delta f\|_{L_p(\{-1,1\}^n;X)} \leq (1-\eta)d^2 \|f\|_{L_p(\{-1,1\}^n;X)}.
\end{equation}
Then $X$ has finite cotype.
\end{theorem}


\subsection{Estimates for $K$-convex Banach spaces} \label{subsec:1.2} Let $(X,\|\cdot\|_X)$ be a Banach space. The $X$-valued Rademacher projection $\msf{Rad}:L_p(\{-1,1\}^n;X)\to L_p(\{-1,1\}^n;X)$ is the operator given by
\begin{equation}
\forall \ \e\in\{-1,1\}^n, \ \ \ \msf{Rad}(f)(\e) \eqdef \sum_{i=1}^n \widehat{f}(\{i\}) \e_i.
\end{equation}
We say that $(X,\|\cdot\|_X)$ is $K$-convex (see also \cite{Mau03}) if $\sup_{n\in\N} \|\msf{Rad}\|_{L_p(\{-1,1\}^n;X)\to L_p(\{-1,1\}^n;X)} < \infty$ for some (equivalently, for all) $p\in(1,\infty)$. Pisier's $K$-convexity theorem \cite{Pis82} asserts that a Banach space $(X,\|\cdot\|_X)$ is $K$-convex if and only if $X$ does not contain copies of $\{\ell_1^n\}_{n=1}^\infty$ with distortion arbitrarily close to 1. Moreover, both conditions are equivalent to $X$ having nontrivial type, see \cite{Pis73}. In the proofs of most results of this section, we will crucially use the deep fact that the heat semigroup with values in a $K$-convex space is a bounded analytic semigroup \cite{Pis82}.

\subsubsection{The heat smoothing conjecture} In \cite{MN14}, Mendel and Naor asked whether for every $K$-convex Banach space $(X,\|\cdot\|_X)$ and $p\in(1,\infty)$ there exist $c(p,X), C(p,X)\in(0,\infty)$ such that for every $n\in\N$, every function $f:\{-1,1\}^n\to X$ in the $d$-th tail space satisfies the estimate
\begin{equation}
\forall \ t\geq0, \ \ \ \|e^{-t\Delta}f\|_{L_p(\{-1,1\}^n;X)} \leq C(p,X)e^{-c(p,X)dt} \|f\|_{L_p(\{-1,1\}^n;X)}.
\end{equation}
In the direction of this question, currently known as the heat smoothing conjecture, they showed the following theorem, partially relying on ideas from \cite{Pis07} (see also the work \cite{HMO17} of Heilman, Mossel and Oleszkiewicz for an optimal result when $d=1$ and $X=\C$).

\begin{theorem} [Mendel--Naor] \label{thm:mendel-naor}
Let $(X,\|\cdot\|_X)$ be a $K$-convex Banach space. For every $p\in(1,\infty)$, there exist $c(p,X), C(p,X)\in(0,\infty)$ and $A(p,X)\in[1,\infty)$ such that for every $n,d\in\N$ with $d\in\{0,1,\ldots,n-1\}$ and every function $f:\{-1,1\}^n\to X$ in the $d$-th tail space, we have
\begin{equation} \label{eq:mendel-naor}
\forall \ t\geq0, \ \ \ \|e^{-t\Delta}f\|_{L_p(\{-1,1\}^n;X)} \leq C(p,X)e^{-c(p,X)d\min\{t,t^{A(p,X)}\}} \|f\|_{L_p(\{-1,1\}^n;X)}.
\end{equation}
\end{theorem}

\noindent In Section \ref{sec:3} below, we present a simple proof of Theorem \ref{thm:mendel-naor} relying on a duality argument.

\subsubsection{An improved lower bound on the decay of the heat semigroup} Under the assumption that the target space $(X,\|\cdot\|_X)$ is $K$-convex, we get the following improvement over Theorem \ref{thm:reverseheatgeneral} (see also equation \eqref{eq:chsqrt} in Remark \ref{rem:ole} for comparison).

\begin{theorem} \label{thm:reverseheatKconvex}
Let $(X,\|\cdot\|_X)$ be a $K$-convex Banach space. For every $p\in(1,\infty)$, there exist $c(p,X), C(p,X)\in(0,\infty)$ and $\eta(p,X) \in \big(\frac{1}{2},1\big]$ such that for every $n,d\in\N$ with $d\in\{1,\ldots,n\}$ and every function $f:\{-1,1\}^n \to X$ of degree at most $d$, we have
\begin{equation} \label{eq:reverseheatKconvex}
\forall \ t\geq0, \ \ \ \|e^{-t\Delta} f\|_{L_p(\{-1,1\}^n;X)} \geq c(p,X) e^{-C(p,X)d\max\{t, t^{\eta(p,X)}\}}\|f\|_{L_p(\{-1,1\}^n;X)}.
\end{equation}
\end{theorem}

Theorem \ref{thm:reverseheatKconvex} should be understood as the dual of Mendel and Naor's bound \eqref{eq:mendel-naor} for the heat smoothing conjecture. We conjecture that one can in fact take $\eta(p,X)=1$ in Theorem \ref{thm:reverseheatKconvex} for every $p\in(1,\infty)$ and $K$-convex Banach space $(X,\|\cdot\|_X)$, but a proof of such a claim appears intractable with the technique presented here. A real valued version of this conjecture for $p=1$ has previously appeared in \cite[Section~5]{Simons}. As in the case of Theorem \ref{thm:reverseheatgeneral}, Theorem \ref{thm:reverseheatKconvex} also implies moment comparison for functions of low degree. We postpone the relevant result\mbox{ (Corollary \ref{cor:momentcompKconvex}) to Section \ref{sec:3}.}

\subsubsection{Bernstein--Markov type inequalities for the hypercube Laplacian and the gradient} Under the additional assumption that the target space $(X,\|\cdot\|_X)$ is $K$-convex, we can obtain the following asymptotic improvement of the bound of Theorem \ref{thm:laplaciangeneral}.

\begin{theorem} \label{thm:laplacianKconvex}
Let $(X,\|\cdot\|_X)$ be a $K$-convex Banach space. For every $p\in(1,\infty)$, there exist $\alpha(p,X) \in [1,2)$ and $C(p,X)\in(0,\infty)$ such that for every $n,d\in\N$ with $d\in\{1,\ldots,n\}$ and every function $f:\{-1,1\}^n \to X$ of degree at most $d$, we have
\begin{equation} \label{eq:laplacianKconvex}
\|\Delta f\|_{L_p(\{-1,1\}^n;X)} \leq C(p,X) d^{\alpha(p,X)} \|f\|_{L_p(\{-1,1\}^n;X)}.
\end{equation}
\end{theorem}

We conjecture that the conclusion of Theorem \ref{thm:laplacianKconvex} holds true with $\alpha(p,X)=1$ for every $p\in(1,\infty)$ and $K$-convex Banach space $(X,\|\cdot\|_X)$ (see also Remark \ref{rem:laplacianimpliesreverseheat} below). Since every $K$-convex Banach space has finite cotype, the conclusion of Theorem \ref{thm:laplacianKconvex} is consistent with Theorem \ref{thm:cotype}. However, there exist Banach spaces (e.g.~$X=\ell_1$) which have finite cotype but are not $K$-convex. It remains an interesting (and potentially challenging) open problem to understand whether the dependence on the degree in the vector valued Bernstein--Markov inequality \eqref{eq:laplaciangeneral} (even for, say, $p=2$) can be improved to $o(d^2)$ under the minimal assumption that $(X,\|\cdot\|_X)$ has finite cotype.

We conclude this section by presenting a Bernstein--Markov type inequality where the hypercube Laplacian is replaced by the vector valued gradient appearing in Pisier's inequality \eqref{eq:pisier}.

\begin{theorem} \label{thm:gradientKconvex}
Let $(X,\|\cdot\|_X)$ be a $K$-convex Banach space. For every $p\in(1,\infty)$ there exist $\alpha(p,X) \in [1,2)$ and $C(p,X)\in(0,\infty)$ such that for every $n,d\in\N$ with $d\in\{1,\ldots,n\}$ and every function $f:\{-1,1\}^n \to X$ of degree at most $d$, we have
\begin{equation} \label{eq:gradientKconvex}
\Big( \frac{1}{2^n} \sum_{\delta\in\{-1,1\}^n} \Big\| \sum_{i=1}^n \delta_i \partial_i f\Big\|_{L_p(\{-1,1\}^n;X)}^p \Big)^{1/p} \leq C(p,X) d^{\alpha(p,X)} (\log d+1) \|f\|_{L_p(\{-1,1\}^n;X)}.
\end{equation}
\end{theorem}

\subsubsection{A reverse Bernstein--Markov inequality for $\Delta$.} In \cite{MN14}, Mendel and Naor asked if for every $K$-convex Banach space $(X,\|\cdot\|_X)$ and $p\in(1,\infty)$ there exists $c(p,X)\in(0,\infty)$ such that for every $n\in\N$ and $d\in\{0,1,\ldots,n-1\}$, every function $f:\{-1,1\}^n\to X$ \mbox{in the $d$-th tail space satisfies}
\begin{equation}
\|\Delta f\|_{L_p(\{-1,1\}^n;X)} \geq c(p,X) d \|f\|_{L_p(\{-1,1\}^n;X)}.
\end{equation}
Similar estimates for scalar valued functions had also been obtained by Meyer in \cite{Mey84}. The following theorem, whose proof relies on a recent inequality of Erd\'elyi, contains a result in this direction under the additional assumption that \mbox{the spectrum of the function $f$ is also bounded above.}

\begin{theorem} \label{thm:reversebernsteinKconvex}
Let $(X,\|\cdot\|_X)$ be a $K$-convex Banach space. For every $p\in(1,\infty)$ there exists $c(p,X)\in(0,\infty)$ such that for every $n,d,m\in\N$ with $d+m\leq n$ and every function $f:\{-1,1\}^n\to X$ of degree at most $d+m$ which is also in the $d$-th tail space, we have
\begin{equation} \label{eq:reverseLaplacian}
\|\Delta f\|_{L_p(\{-1,1\}^n;X)} \geq c(p,X) \frac{d}{m} \|f\|_{L_p(\{-1,1\}^n;X)}.
\end{equation}
\end{theorem}

In particular, Theorem \ref{thm:reversebernsteinKconvex} provides a positive answer to the question of \cite{MN14} in the special case when $m=O(1)$ and improves upon the previously known bounds of Meyer \cite{Mey84} and Mendel and Naor \cite{MN14} when $m$ is a small enough power of $d$.


\subsection{Estimates for scalar valued functions} \label{subsec:1.3} In this section we will present explicit estimates which hold true for scalar valued functions. Even though several of the following results are special cases of theorems from the previous section, we present the full statements for the convenience of the reader not interested in the general vector valued setting. 

\subsubsection{The decay of the heat semigroup} The optimal bounds that can be derived from our approach for the action of the heat semigroup on functions with bounded spectrum are the following.

\begin{theorem} \label{thm:smoothingR}
For $p\in[1,\infty]$, let $\theta_p = 2\arcsin\big(\frac{2\sqrt{p-1}}{p}\big)$. Then, for every $p\in(1,\infty)$, $n,d\in\N$ with $d\in\{0,1,\ldots,n\}$ and every function $f:\{-1,1\}^n\to\C$ of degree at most $d$, we have
\begin{equation} \label{eq:lowersmoothingR}
\forall \ t\geq0, \ \ \ \|e^{-t\Delta}f\|_{L_p(\{-1,1\}^n;\C)} \geq \left(\frac{(e^{t}+1)^{\frac{\pi}{2\pi-\theta_p}} - (e^{t}-1)^{\frac{\pi}{2\pi-\theta_p}}}{(e^{t}+1)^{\frac{\pi}{2\pi-\theta_p}} + (e^{t}-1)^{\frac{\pi}{2\pi-\theta_p}}}\right)^d \|f\|_{L_p(\{-1,1\}^n;\C)}.
\end{equation}
Moreover, for every function $f:\{-1,1\}^n\to\C$ in the $d$-th tail space, we have
\begin{equation} \label{eq:uppersmoothingR}
\forall \ t\geq0, \ \ \ \|e^{-t\Delta}f\|_{L_p(\{-1,1\}^n;\C)} \leq \left(\frac{(1+e^{-t})^{\frac{\pi}{\theta_p}} - (1-e^{-t})^{\frac{\pi}{\theta_p}}}{(1+e^{-t})^{\frac{\pi}{\theta_p}} + (1-e^{-t})^{\frac{\pi}{\theta_p}}}\right)^d \|f\|_{L_p(\{-1,1\}^n;\C)}.
\end{equation}
\end{theorem}


The above lower bound \eqref{eq:lowersmoothingR} on the decay of the heat semigroup combined with classical hypercontractivite estimates \cite{Bon70} implies the following improved moment comparison result for functions of low degree.

\begin{corollary} \label{cor:momentcompR} 
For $p\in[1,\infty]$, let $\theta_p = 2\arcsin\big(\frac{2\sqrt{p-1}}{p}\big)$. Then, for every $p>q>1$, every $n,d\in\N$ with $d\in\{1,\ldots,n\}$ and every function $f:\{-1,1\}^n\to\C$ of degree at most $d$, we have
\begin{equation}
 \|f\|_{L_p(\{-1,1\}^n;\C)} \leq \left(\frac{(\sqrt{p-1}+\sqrt{q-1})^{\frac{\pi}{2\pi-\theta_p}} + (\sqrt{p-1}-\sqrt{q-1})^{\frac{\pi}{2\pi-\theta_p}}}{(\sqrt{p-1}+\sqrt{q-1})^{\frac{\pi}{2\pi-\theta_p}} - (\sqrt{p-1}-\sqrt{q-1})^{\frac{\pi}{2\pi-\theta_p}}}\right)^d \|f\|_{L_q(\{-1,1\}^n;\C)}.
\end{equation}
\end{corollary}

Even though Bonami's hypercontractive estimates \cite{Bon70} break down at the endpoint $p=1$, a well known trick (see \cite[Theorem~9.22]{O'D14})\mbox{ implies that if $f:\{-1,1\}^n\to \C$ has degree at most $d$,}
\begin{equation}
\|f\|_{L_2(\{-1,1\}^n;\C)} \leq e^d \|f\|_{L_1(\{-1,1\}^n;\C)}.
\end{equation}
Relying on works of Beckner \cite{Bec75} and Weissler \cite{Wei79}, we\mbox{ prove the following improved bound.}

\begin{theorem} \label{thm:momentcomp1}
For every $n,d\in\N$ with $d\in\{1,\ldots,n\}$ and every function $f:\{-1,1\}^n\to\C$ of degree at most $d$, we have
\begin{equation} \label{eq:momentcomp1}
\|f\|_{L_2(\{-1,1\}^n;\C)} \leq (2.69076)^d \|f\|_{L_1(\{-1,1\}^n;\C)}.
\end{equation}
\end{theorem}

\subsubsection{Bernstein--Markov type inequalities} The bounds \eqref{eq:laplacianKconvex} in the Bernstein--Markov inequality for the hypercube Laplacian take the following explicit form for scalar valued functions.

\begin{theorem} \label{thm:laplacianR}
For $p\in[1,\infty]$, let $\theta_p = 2\arcsin\big(\frac{2\sqrt{p-1}}{p}\big)$. Then, for every $n,d\in\N$ with $d\in\{1,\ldots,n\}$ and every function $f:\{-1,1\}^n\to\C$ of degree at most $d$, we have
\begin{equation} \label{eq:gooddomain}
\|\Delta f\|_{L_p(\{-1,1\}^n;\C)} \leq 10 d^{2-\frac{\theta_p}{\pi}} \|f\|_{L_p(\{-1,1\}^n;\C)}.
\end{equation}
\end{theorem}

For $p\in[1,\infty]$ and a function $f:\{-1,1\}^n\to\C$, denote by
\begin{equation}
\|\nabla f\|_{L_p(\{-1,1\}^n;\C)} \eqdef \Big\| \Big(\sum_{i=1}^n (\partial_i f)^2\Big)^{1/2} \Big\|_{L_p(\{-1,1\}^n;\C)}.
\end{equation}
In contrast to the vector valued Theorem \ref{thm:gradientKconvex}, we can prove the following improved Bernstein--Markov type inequality for the gradient of scalar valued functions as a consequence of Theorem \ref{thm:laplacianR}.

\begin{theorem} \label{thm:gradR}
For $p\in[1,\infty]$, let $\theta_p = 2\arcsin\big(\frac{2\sqrt{p-1}}{p}\big)$. Then, for every $p\in(1,\infty)$, there exists $C_p\in(0,\infty)$ such that for every $n,d\in\N$ with $d\in\{1,\ldots,n\}$ and every function $f:\{-1,1\}^n\to\C$ of degree at most $d$, we have
\begin{equation}
\|\nabla f\|_{L_p(\{-1,1\}^n;\C)} \leq   \begin{cases} 
C_p d^{\frac{2}{p}-\frac{\theta_p}{p\pi}} \log (d+1) \|f\|_{L_p(\{-1,1\}^n;\C)}, & \mbox{if } p\in\big(1,2\big)
\\ C_pd^{1-\frac{\theta_p}{2\pi}} \|f\|_{L_p(\{-1,1\}^n;\C)}, & \mbox{if } p\in[2,\infty) \end{cases}.
\end{equation}
\end{theorem}

The change in the exponent from the range $p\in(1,2)$ to the range $p\in[2,\infty)$ and its relation to discrete Riesz transforms is further discussed in Section \ref{sec:4}. We also postpone until then the statement of some endpoint ($p=\infty$ and $p=1$) Bernstein--Markov inequalities for the discrete gradient (see Proposition \ref{prop:gradp=infty} and Remark \ref{rem:gradp=1} respectively).

We finally turn to a problem studied by Filmus, Hatami, Keller and Lifshitz in \cite{FHKL16} (see also \cite{BB14}). Following their notation, for $p\in[1,\infty)$ and a function $f:\{-1,1\}^n\to\C$, denote by
\begin{equation}
\mathrm{Inf}^{(p)} f \eqdef \sum_{i=1}^n \|\partial_i f\|^p_{L_p(\{-1,1\}^n;\C)}.
\end{equation}
Motivated by a question of Aaronson and Ambrainis \cite{AA14}, the authors were interested in obtaining bounds of the form 
\begin{equation} \label{eq:influences}
\mathrm{Inf}^{(p)} f \leq f_p(d) \|f\|_{L_\infty(\{-1,1\}^n;\C)}^p,
\end{equation}
where $f_p(d)$ is a polynomial in $d$. In this direction, they proved \eqref{eq:influences} with $f_p(d) = d^{3-p}$ when $p\in[1,2]$ and $f_p(d)=d$ when $p\in[2,\infty)$, the latter of which is sharp. Using Theorem \ref{thm:smoothingR}, we deduce the following improved bounds for $p\in\big(1,\frac{4}{3}\big)$.

\begin{corollary} \label{cor:influences}
For every $p\in\big(1,\frac{4}{3}\big)$ there exists a constant $K_p\in(0,\infty)$ such that for every $n,d\in\N$ with $d\in\{1,\ldots,n\}$ and every function $f:\{-1,1\}^n\to\C$ of degree at most $d$, we have
\begin{equation} \label{eq:influencebound}
\mathrm{Inf}^{(p)} f \leq K_pd^{2-\frac{1}{\pi}\arcsin\big(\frac{2\sqrt{p-1}}{p}\big)} \|f\|_{L_\infty(\{-1,1\}^n;\C)}^p.
\end{equation}
\end{corollary}

\subsubsection{A reverse Bernstein--Markov type inequality for $\nabla$} For the case of scalar valued functions we will also prove the following variant of Theorem \ref{thm:reversebernsteinKconvex} for the discrete gradient.

\begin{theorem} \label{thm:reversebernstein}
For every $p\in(1,\infty)$ there exists $c_p\in(0,\infty)$ such that for every $n,d,m\in\N$ with $d+m\leq n$ and every function $f:\{-1,1\}^n\to\C$ of degree at most $d+m$ which is also in the $d$-th tail space, we have
\begin{equation} \label{eq:reversegradient}
\|\nabla f\|_{L_p(\{-1,1\}^n;\C)} \geq c_p \sqrt{\frac{d}{m}} \|f\|_{L_p(\{-1,1\}^n;\C)}.
\end{equation}
\end{theorem}

\subsection*{Acknowledgements} We are indebted to Assaf Naor for many helpful discussions. We are also very grateful to Tam\'as Erd\'elyi for proving the main result of \cite{Erd18} upon our request and to an anonymous referee for sharing with us an argument which improved Corollary \ref{cor:entropy}. Finally, we would like to thank Françoise Lust-Piquard for valuable feedback.


\section{Estimates for a general Banach space} \label{sec:2}

We first present the proof of Theorem \ref{thm:reverseheatgeneral}, the lower bound on the decay of the heat semigroup acting on functions with values in a general Banach space. Recall that the $d$-th Chebyshev polynomial of the first kind $T_d(x)$ is the unique polynomial of degree $d$ such that $T_d(\cos \theta) = \cos(d\theta)$ for every $\theta\in\R$. Chebyshev's inequality \cite[p.~235]{BE95} asserts that the $d$-th Chebyshev polynomial of the first kind is characterized by the extremal property
\begin{equation} \label{eq:chebyshevineq}
\forall \ x\in\R\setminus[-1,1], \ \ \ |T_d(x)| = \max\big\{ |p(x)|: \ \mathrm{deg}(p) \leq d \ \mbox{and} \ \|p\|_{\ms{C}([-1,1])}=1 \big\},
\end{equation}
where for a continuous function $h:K\to\C$ on a compact space $K$, we set $\|h\|_{\ms{C}(K)} = \max_{x\in K} |h(x)|$.

\medskip

\noindent {\it Proof of Theorem \ref{thm:reverseheatgeneral}.}
Fix $n\in\N$, $d\in\{1,\ldots,n\}$, $p\in[1,\infty]$ and let $f:\{-1,1\}^n\to X$ be a function of degree at most $d$. The contractivity of the heat semigroup which we derived from \eqref{eq:heatrepresent}, can be rewritten as $\|x^\Delta f\|_{L_p(\{-1,1\}^n;X)} \leq \|f\|_{L_p(\{-1,1\}^n;X)}$ for every $x\in[0,1]$. However, for $x\in\R$,
\begin{equation} \label{eq:symmetrysemigroup}
(-x)^\Delta f(\e) = \sum_{A\subseteq\{1,\ldots,n\}} (-x)^{|A|} \widehat{f}(A) w_A(\e) =  \sum_{A\subseteq\{1,\ldots,n\}} x^{|A|} \widehat{f}(A) w_A(-\e) = (x^\Delta f)(-\e),
\end{equation}
thus $\|(-x)^\Delta f\|_{L_p(\{-1,1\}^n;X)} = \|x^\Delta f\|_{L_p(\{-1,1\}^n;X)}$. Consequently,
\begin{equation} \label{eq:contract<0}
\forall \ x\in[-1,1], \ \ \ \|x^\Delta f\|_{L_p(\{-1,1\}^n;X)}\leq \|f\|_{L_p(\{-1,1\}^n;X)}.
\end{equation}
Let $\mu$ be any complex measure on $[-1,1]$. Then, averaging over \eqref{eq:contract<0}, we get
\begin{equation} \label{eq:usetriangleformeasure}
\begin{split}
\Big\| \int_{-1}^1 x^\Delta f \diff\mu(x)&\Big\|_{L_p(\{-1,1\}^n;X)} \leq \int_{-1}^1 \|x^\Delta f\|_{L_p(\{-1,1\}^n;X)} \diff|\mu|(x)
\\ & \stackrel{\eqref{eq:contract<0}}{\leq} \int_{-1}^1 \|f\|_{L_p(\{-1,1\}^n;X)} \diff|\mu|(x) = \|\mu\|_{\ms{M}([-1,1])} \|f\|_{L_p(\{-1,1\}^n;X)},
\end{split}
\end{equation}
where for a complex measure $\mu$ on a compact space $K$, we denote by $\|\mu\|_{\ms{M}(K)}$ the total variation of $\mu$. For $t\geq0$, consider the linear \mbox{functional $\varphi_t: \big( \mathrm{span}\{1,x,\ldots,x^d\}, \|\cdot\|_{\ms{C}([-1,1])}\big) \to \C$ given by}
\begin{equation}
\varphi_t\Big(\sum_{k=0}^d a_k x^k\Big) \eqdef \sum_{k=0}^d a_ke^{tk},
\end{equation}
or $\varphi_t(p) = p(e^t)$ when $p$ is a polynomial with $\mathrm{deg}(p)\leq d$. Then, Chebyshev's inequality \eqref{eq:chebyshevineq} can be rewritten as
\begin{equation} \label{cheche}
\forall \ p\in\mathrm{span}\{1,x,\ldots,x^d\}, \ \ \ |\varphi_t(p)| \leq T_d(e^t) \|p\|_{\ms{C}([-1,1])}.
\end{equation}
Therefore, by the Hahn--Banach theorem and the Riesz representation theorem, there exists a complex measure $\mu_t$ on $[-1,1]$ such that $\|\mu_t\|_{\ms{M}([-1,1])}\leq T_d(e^t)$ and for every polynomial $p$, we have
\begin{equation} \label{eq:constructedmeasure}
\mathrm{deg}(p)\leq d \ \ \ \Longrightarrow \ \ \ \int_{-1}^1 p(x)\diff\mu_t(x) = p(e^t).
\end{equation}
Since $\widehat{f}(A)=0$ when $|A|>d$, applying \eqref{eq:usetriangleformeasure} for the measure $\mu_t$ satisfying \eqref{eq:constructedmeasure}, we deduce that
\begin{equation}
\forall \ t\geq0, \ \ \ \|e^{t\Delta} f\|_{L_p(\{-1,1\}^n;X)} \stackrel{\eqref{eq:constructedmeasure}}{=} \Big\| \int_{-1}^1 x^\Delta f \diff\mu_t(x)\Big\|_{L_p(\{-1,1\}^n;X)} \stackrel{\eqref{cheche}}{\leq} T_d(e^t) \|f\|_{L_p(\{-1,1\}^n;X)},
\end{equation}
which is equivalent to the desired inequality \eqref{eq:reverseheatgeneral}.
\hfill$\Box$

\begin{remark} \label{rem:ole}
A weaker lower bound for the decay of the heat semigroup, attributed partially to Oleszkiewicz, was established in \cite[Lemma~5.4]{FHKL16} and asserts that for every function $f:\{-1,1\}^n \to X$ of degree at most $d$, we have
\begin{equation} \label{eq:ole}
\|e^{-t\Delta}f\|_{L_p(\{-1,1\}^n;X)} \geq e^{-td^2} \|f\|_{L_p(\{-1,1\}^n;X)}.
\end{equation}
Inequality \eqref{eq:ole} was stated in \cite{FHKL16} only for $X=\R$ and $p=1$, but the argument presented there works in greater generality. Checking that \eqref{eq:ole} is weaker than the estimate \eqref{eq:reverseheatgeneral} amounts to showing that for $y\geq1$, we have $T_d(y)\leq y^{d^2}$, which can be easily derived from the identity
\begin{equation} \label{eq:chebyshevidentity}
\forall \ |y|\geq1, \ \ \ T_d(y) = \frac{(y+\sqrt{y^2-1})^d+(y-\sqrt{y^2-1})^d}{2}.
\end{equation}
Furthermore, using \eqref{eq:chebyshevidentity}, it is elementary to check that
\begin{equation} \label{eq:chsqrt}
\forall \ t\geq0, \ \ \ T_d(e^t) \leq e^{3\max\{t,\sqrt{t}\}d},
\end{equation}
therefore inequality \eqref{eq:reverseheatKconvex} is also an improvement over \eqref{eq:reverseheatgeneral} for $t\approx0$.
\end{remark}

\smallskip

\noindent {\it Proof of Corollary \ref{cor:realchaos}.}
Fix $n\in\N$, $d\in\{1,\ldots,n\}$, $p>q>1$ and a function $f:\{-1,1\}^n\to X$ of degree at most $d$. Bonami's hypercontractive inequality \cite{Bon70} (the straightforward vector valued extension of which is due to Borell, see \cite{Bor79}) asserts that
\begin{equation} \label{eq:bonami}
\forall \ t\geq \frac{1}{2}\log\Big(\frac{p-1}{q-1}\Big), \ \ \ \|e^{-t\Delta} f\|_{L_p(\{-1,1\}^n;X)} \leq \|f\|_{L_q(\{-1,1\}^n;X)}.
\end{equation}
Combining \eqref{eq:reverseheatgeneral} and \eqref{eq:bonami}, we deduce that for $t\geq\frac{1}{2}\log\big(\frac{p-1}{q-1}\big)$,
\begin{equation} \label{eq:combinehyperandthm}
\|f\|_{L_p(\{-1,1\}^n;X)} \stackrel{\eqref{eq:reverseheatgeneral}}{\leq} T_d(e^t) \|e^{-t\Delta}f\|_{L_p(\{-1,1\}^n;X)} \stackrel{\eqref{eq:bonami}}{\leq} T_d(e^t) \|f\|_{L_q(\{-1,1\}^n;X)}.
\end{equation}
Plugging $t=\frac{1}{2}\log\big(\frac{p-1}{q-1}\big)$ in \eqref{eq:combinehyperandthm}, we deduce the moment comparison \eqref{eq:realchaos}.
\hfill$\Box$

\begin{remark} \label{rem:kwapien}
Inequality \eqref{eq:realchaos} for $d=1$ with some constant $C(p,q)$ originated in the work \cite{Kah64} of Kahane and the implicit constant was later improved to $\sqrt{\frac{p-1}{q-1}}$ by Kwapie\'{n} in \cite{Kwa76}. The use of hypercontractivity for moment comparison of vector valued Walsh polynomials was initiated by Borell in \cite{Bor79} (see the exposition \cite{Pis78}), who later showed in \cite{Bor84} that inequality \eqref{eq:realchaos} holds true with some constant depending only on $p,q$ and $d$ (this had previously been proven for scalar valued functions by Bourgain in \cite{Bou80} via a square function approach). To the extent of our knowledge, the best known dependence in \eqref{eq:realchaos} before the present work could be extracted from an argument in the monograph \cite[Proposition~6.5.1]{KW92} which goes as follows. For $k\in\{0,1,\ldots,n\}$ let $\msf{Rad}_k:L_p(\{-1,1\}^n;X)\to L_p(\{-1,1\}^n;X)$ be the Rademacher projection on level $k$ given by
\begin{equation}
\forall \ \e\in\{-1,1\}^n, \ \ \ \msf{Rad}_k(f)(\e) \eqdef \sum_{\substack{A\subseteq\{1,\ldots,n\}\\ |A|=k}} \widehat{f}(A) w_A(\e).
\end{equation}
If $d\in\{0,1,\ldots,n\}$, consider the linear functional $\sigma_k: \big( \mathrm{span}\{1,x,\ldots,x^d\}, \|\cdot\|_{\ms{C}([-1,1])}\big) \to \C$ given by $\sigma_k(p) = \frac{p^{(k)}(0)}{k!}$. Then, for every $k\in\{1,\ldots,d\}$, there exists some constant \mbox{$c(d,k)\in(0,\infty)$ such that}
\begin{equation}
\mathrm{deg}(p)\leq d \ \ \ \Longrightarrow \ \ \ |\sigma_k(p)| \leq c(d,k) \|p\|_{\ms{C}([-1,1])}.
\end{equation}
The optimal value of $c(d,k)$ is known, see \cite[p.~248]{BE95}, but clearly $c(d,k)\geq \frac{|T_d^{(k)}(0)|}{k!}$. Repeating the duality argument used in the proof of Theorem \ref{thm:reverseheatgeneral}, we deduce that for every function $f:\{-1,1\}^n\to X$ of degree at most $d$, $q\in[1,\infty]$ and $k\in\{0,\ldots,d\}$,
\begin{equation}\label{eq:radbounded}
\big\|\msf{Rad}_k f\big\|_{L_q(\{-1,1\}^n;X)} \leq c(d,k) \|f\|_{L_q(\{-1,1\}^n;X)}.
\end{equation}
Furthermore, by the moment comparison of homogeneous Walsh polynomials \cite{Bor79}, for $p>q>1$,
\begin{equation} \label{eq:borellhomogeneous}
\big\|\msf{Rad}_k f\big\|_{L_p(\{-1,1\}^n;X)} \leq \Big(\frac{p-1}{q-1}\Big)^{k/2} \big\|\msf{Rad}_k f\big\|_{L_q(\{-1,1\}^n;X)}.
\end{equation}
Therefore, combining \eqref{eq:radbounded} and \eqref{eq:borellhomogeneous}, we get
\begin{equation}
\begin{split}
\|f\|_{L_p(\{-1,1\}^n;X)} \leq \sum_{k=0}^d \big\| \msf{Rad}_k f\big\|_{L_p(\{-1,1\}^n;X)}  \stackrel{\eqref{eq:borellhomogeneous}}{\leq} &\sum_{k=0}^d  \Big(\frac{p-1}{q-1}\Big)^{k/2}  \big\|\msf{Rad}_k f\big\|_{L_q(\{-1,1\}^n;X)} \\ & \stackrel{\eqref{eq:radbounded}}{\leq} \sum_{k=0}^d c(d,k) \Big(\frac{p-1}{q-1}\Big)^{k/2}  \|f\|_{L_q(\{-1,1\}^n;X)}.
\end{split}
\end{equation}
However, the constant obtained by this argument is
\begin{equation}
\begin{split}
\sum_{k=0}^d c(d,k) \Big(\frac{p-1}{q-1}\Big)^{k/2} \geq \sum_{k=0}^d \frac{|T_d^{(k)}(0)|}{k!} \Big(\frac{p-1}{q-1}\Big)^{k/2} \geq  \sum_{k=0}^d \frac{T_d^{(k)}(0)}{k!} \Big(\frac{p-1}{q-1}\Big)^{k/2}  = T_d\Big(\sqrt{\frac{p-1}{q-1}}\Big),
\end{split}
\end{equation}
and therefore Corollary \ref{cor:realchaos} improves over the result of \cite{KW92}. 
\end{remark}

Combining Corollary \ref{cor:realchaos} and \eqref{eq:chsqrt}, it follows that for $p>q>1$ and every function $f:\{-1,1\}^n\to X$ of degree at most $d$, we have
\begin{equation}
\|f\|_{L_p(\{-1,1\}^n;X)} \leq e^{3d\max\big\{ \frac{1}{2}\log\big(\frac{p-1}{q-1}\big), \frac{1}{\sqrt{2}}\log \big(\frac{p-1}{q-1} \big)^{1/2} \big\}}\|f\|_{L_q(\{-1,1\}^n;X)}.
\end{equation}
In particular, if $p-1 \geq (1+\e)(q-1)$ for some $\e>0$, then there exists $C_\e=O(1/\sqrt{\e})>0$ such that every function $f:\{-1,1\}^n\to X$ of degree at most $d$ satisfies
\begin{equation} \label{eq:momentcompKconvex}
\|f\|_{L_p(\{-1,1\}^n;X)} \leq \Big( \frac{p-1}{q-1} \Big)^{C_\e d}\|f\|_{L_q(\{-1,1\}^n;X)},
\end{equation}
which should be understood as an extension of Borell's moment comparison \eqref{eq:borellhomogeneous} for vector valued functions of low degree instead of homogeneous. The validity of Theorem \ref{thm:reverseheatKconvex} with $\eta(p,X)=1$ would imply the estimate \eqref{eq:momentcompKconvex} with $C_\e$ replaced by a universal contant $C\in(0,\infty)$.

\begin{remark}
Some similar moment comparison estimates can be obtained for functions on the $\alpha$-biased hypercube using the biased hypercontractivity of Oleszkiewicz \cite{Ole03} and Wolff \cite{Wol07}. Indeed, it follows from the main results of \cite{Ole03, Wol07} and convexity considerations that for every $p>q>1$ and $\alpha\in(0,1)$, there exists a closed interval $I_{p,q}(\alpha) \subseteq[-1,1]$ with nonempty interior containing 0  such that for every $n\in\N$ and every coefficients $\{a_S\}_{S\subseteq\{1,\ldots,n\}}\subset X$,
\begin{equation}
\sup_{x\in I_{p,q}(\alpha)} \Big\| \sum_{S\subseteq\{1,\ldots,n\}} x^{|S|} a_S w_S \Big\|_{L_p(\{-\gamma^{-1},\gamma\}^n, \mu_\alpha^n ;X)} \leq \Big\| \sum_{S\subseteq\{1,\ldots,n\}} a_S w_S \Big\|_{L_q(\{-\gamma^{-1},\gamma\}^n, \mu_\alpha^n ;X)},
\end{equation}
where $\mu_\alpha = \alpha \delta_{-\gamma^{-1}} + (1-\alpha) \delta_\gamma$ and $\gamma = \sqrt{\tfrac{1}{\alpha}-1}$. Running the same duality argument as in the proof of Theorem \ref{thm:reverseheatgeneral} combined with an application of Chebyshev's inequality \eqref{eq:chebyshevineq} for a symmetric interval $J_{p,q}(\alpha) \subseteq I_{p,q}(\alpha)$, proves an analogue of \eqref{eq:realchaos} on $\big(\{-\gamma^{-1},\gamma\}^n, \mu_\alpha^n\big)$.
\end{remark}

Finally we state one ``endpoint" version\footnote{The bound \eqref{eq:entropy} of Corollary \ref{cor:entropy} was pointed out to us by an anonymous referee, who improved a suboptimal $O_q(d^2)$ estimate appearing in an earlier version of this manuscript. We are grateful to them for sharing their improvement with us and for their helpful comments.} of our moment comparison \eqref{eq:realchaos}. Recall that for a function $h:\{-1,1\}^n\to[0,\infty)$, we denote its entropy by
\begin{equation}
\mathrm{Ent}(h) \eqdef \frac{1}{2^n}\sum_{\e\in\{-1,1\}^n} h(\e)\log h(\e)- \Big(\frac{1}{2^n}\sum_{\e\in\{-1,1\}^n} h(\e)\Big) \cdot \log\Big(\frac{1}{2^n}\sum_{\e\in\{-1,1\}^n}h(\e)\Big).
\end{equation}

\begin{corollary} \label{cor:entropy}
Fix $n,d\in\N$ with $d\in\{1,\ldots,n\}$ and let $(X,\|\cdot\|_X)$ be a Banach space. For every $q\in(0,\infty)$ there exists a constant $C_q\in(0,\infty)$ such that every function $f:\{-1,1\}^n\to X$ of degree at most $d$ satisfies
\begin{equation} \label{eq:entropy}
\mathrm{Ent}( \|f\|_X^q) \leq C_qd \|f\|_{L_q(\{-1,1\}^n;X)}^q.
\end{equation}
\end{corollary}

\begin{proof}
By H\"older's inequality, for every $\alpha,\beta\in(0,\infty)$ and $\lambda\in(0,1)$,
\begin{equation} \label{eq:fancyholder}
\|f\|_{L_{(\frac{\lambda}{\alpha}+\frac{1-\lambda}{\beta})^{-1}}(\{-1,1\}^n;X)} \leq \|f\|_{L_\alpha(\{-1,1\}^n;X)}^\lambda \|f\|_{L_\beta(\{-1,1\}^n;X)}^{1-\lambda}.
\end{equation}
In other words, the function $\upphi:(0,\infty)\to\R$ given by
\begin{equation}
\forall \ r\in(0,\infty), \ \ \ \ \upphi(r)\eqdef \log\big(\|f\|_{L_{1/r}(\{-1,1\}^n;X)}\big)
\end{equation}
is convex. Hence, the function $(0,\infty)\ni r\longmapsto \upphi'(1/r)$ is nonincreasing, which implies that
\begin{equation} \label{eq:upphi}
\forall \ p<q, \ \ \ \upphi'\big(\tfrac{1}{q}\big) \cdot \big( \tfrac{1}{p}-\tfrac{1}{q}\big) \geq \upphi\big(\tfrac{1}{p}\big) - \upphi\big(\tfrac{1}{q}\big).
\end{equation}
Writing
\begin{equation}
\|f\|_{L_\alpha(\{-1,1\}^n;X)} = \exp\Big( \frac{1}{\alpha}\log\Big( \frac{1}{2^n}\sum_{\e\in\{-1,1\}^n} \|f(\e)\|_X^\alpha \Big)\Big),
\end{equation}
one deduces that
\begin{equation}
\frac{\diff}{\diff \alpha} \|f\|_{L_\alpha(\{-1,1\}^n;X)} = \|f\|_{L_\alpha(\{-1,1\}^n;X)} \cdot \frac{\mathrm{Ent}\|f\|_X^\alpha}{\alpha^2\|f\|_{L_\alpha(\{-1,1\}^n;X)}^\alpha},
\end{equation}
which implies that
\begin{equation} \label{eq:compderiva}
\upphi'\Big(\frac{1}{q}\Big) = - \frac{\mathrm{Ent}\|f\|_X^q}{\|f\|_{L_q(\{-1,1\}^n;X)}^q}.
\end{equation}
By the monotonicity of $r\mapsto\upphi'(1/r)$ and \eqref{eq:compderiva}, in order to prove \eqref{eq:entropy} for every $q\in(0,\infty)$ it suffices to consider the case $q\geq4$. By \eqref{eq:compderiva}, \eqref{eq:upphi} can be rewritten as
\begin{equation}
\begin{split}
\frac{\mathrm{Ent}\|f\|_X^q}{\|f\|_{L_q(\{-1,1\}^n;X)}^q} & \leq \frac{pq}{q-p} \log\left(\frac{\|f\|_{L_q(\{-1,1\}^n;X)}}{\|f\|_{L_p(\{-1,1\}^n;X)}}\right) 
\\ & \stackrel{\eqref{eq:realchaos}}{\leq} \frac{pq}{q-p} \log\Big(T_d\Big(\sqrt{\frac{q-1}{p-1}}\Big)\Big) \leq \frac{pqd}{2(q-p)} \log \frac{4(q-1)}{p-1},
\end{split}
\end{equation}
where in the last inequality we used the estimate $T_d(y)\leq (2y)^d$ for $y>1$ which follows from \eqref{eq:chebyshevidentity}. Choosing $p=q/2$ completes the proof with $C_q=O(q)$ as $q\to\infty$.
\end{proof}

\begin{remark} \label{rem:momentboost}
The application \eqref{eq:fancyholder} of H\"older's inequality can be rewritten as
\begin{equation} \label{eq:holdertrick}
\forall \ 0<q<p<r, \ \ \ \log\left( \frac{\|f\|_{L_p(\{-1,1\}^n;X)}}{\|f\|_{L_q(\{-1,1\}^n;X)}}\right) \leq \frac{(p-q)r}{(r-q)p} \cdot \log\left( \frac{\|f\|_{L_r(\{-1,1\}^n;X)}}{\|f\|_{L_q(\{-1,1\}^n;X)}}\right).
\end{equation}
It is classical that using \eqref{eq:holdertrick}, one can improve moment comparison bounds for low degree functions (see, e.g., \cite{IT18} for a recent application of this idea). The same applies to the moment comparison inequalities obtained in this paper. For instance, combining \eqref{eq:realchaos} with \eqref{eq:holdertrick}, we deduce that for every function $f:\{-1,1\}^n\to X$ of degree at most $d$ and every $p>q>1$,
\begin{equation} \label{eq:imporvedchaos}
\|f\|_{L_p(\{-1,1\}^n;X)} \leq \inf_{r>p} T_d\Big(\sqrt{\frac{r-1}{q-1}}\Big)^{\frac{(p-q)r}{(r-q)p}} \cdot \|f\|_{L_q(\{-1,1\}^n;X)}.
\end{equation}
Moreover it is well known that the same trick relying on \eqref{eq:holdertrick} can be used to prove moment comparison estimates beyond the range $p>q>1$ (see, e.g.,  \cite[Theorem~9.22]{O'D14} or Theorem \ref{thm:momentcomp1} below). Since all such applications of \eqref{eq:holdertrick} are automatic, in the sequel we will omit stating improvements such as \eqref{eq:imporvedchaos} (for instance, of Corollary \ref{cor:momentcompR} or Corollary \ref{cor:momentcompKconvex}).
\end{remark}

We proceed by proving Pisier's inequality \eqref{eq:Pisierlowfreq} for functions of low degree. The proof is an almost mechanical adaptation of Pisier's argument from \cite{Pis86} (see also the exposition in \cite{Nao12}) with the exception of suitably using the lower bound \eqref{eq:reverseheatgeneral} instead of the trivial lower bound 
\begin{equation}
\forall \ t\geq0, \ \ \ \|e^{-t\Delta}f\|_{L_p(\{-1,1\}^n;X)} \geq e^{-nt}\|f\|_{L_p(\{-1,1\}^n;X)},
\end{equation} 
which holds true for every function $f:\{-1,1\}^n\to X$. For completeness, we\mbox{ present the full proof.}

\medskip

\noindent {\it Proof of Theorem \ref{thm:Pisierlowfreq}.} Fix $p\in[1,\infty)$ and let $f:\{-1,1\}^n\to X$ be such that $\sum_{\delta\in\{-1.1\}^n} f(\delta)=0$. Fix $s\geq0$ and consider a normalizing functional $g_s^\ast$ of $e^{-s\Delta}f$ in $L_p(\{-1,1\}^n;X)$, that is, a function $g_s^\ast:\{-1,1\}^n \to X^\ast$ such that $\|g_s^\ast\|_{L_q(\{-1,1\}^n;X^\ast)}=1$, where $\frac{1}{p}+\frac{1}{q}=1$, and
\begin{equation} \label{eq:writenormfun}
\|e^{-s\Delta}f\|_{L_p(\{-1,1\}^n;X)} = \frac{1}{2^n} \sum_{\e\in\{-1,1\}^n} \langle g_s^\ast(\e), e^{-s\Delta}f(\e)\rangle = \frac{1}{2^n} \sum_{\e\in\{-1,1\}^n} \langle e^{-s\Delta} g_s^\ast(\e), f(\e)\rangle,
\end{equation}
where in the last equality we used the fact that $e^{-s\Delta}$ is self-adjoint. However, for $\e\in\{-1,1\}^n$,
\begin{equation} \label{eq:usenormfun}
\begin{split}
\langle e^{-s\Delta} g_s^\ast(\e), f(\e)\rangle & = \frac{1}{2^n} \sum_{\delta\in\{-1,1\}^n} \langle e^{-s\Delta} \Delta^{-1} \Big( \sum_{i=1}^n \delta_i \partial_i\Big)g_s^\ast(\e) ,\sum_{i=1}^n \delta_i \partial_i f(\e) \rangle
\\ & = \frac{1}{2^n} \sum_{\delta\in\{-1,1\}^n} \int_s^\infty \langle  \sum_{i=1}^n \delta_i \partial_i e^{-t\Delta} g_s^\ast(\e) ,\sum_{i=1}^n \delta_i \partial_i f(\e) \rangle \diff t.
\end{split}
\end{equation}
For $s,t\geq0$, consider the fucntion $g_{s,t}^\ast: \{-1,1\}^n\times \{-1,1\}^n \to X^\ast$ given by
\begin{equation} \label{eq:takeradproj}
\begin{split}
g_{s,t}^\ast(\e,\delta) & \eqdef \sum_{A\subseteq\{1,\ldots,n\}} \widehat{g_s^\ast}(A) \prod_{i\in A} \big( e^{-t} \e_i + (1-e^{-t}) \delta_i\big)
\\ & = e^{-t\Delta} g_s^\ast(\e) + (e^t-1) \sum_{i=1}^n \e_i\delta_i\partial_i e^{-t\Delta}g_s^\ast(\e) + \Phi_{s,t}^\ast(\e,\delta),
\end{split}
\end{equation}
where $\sum_{\delta\in\{-1,1\}^n} \delta_i \Phi_{s,t}^\ast(\e,\delta)=0$ for every $\e\in\{-1,1\}^n$ and $i\in\{1,\ldots,n\}$. Furthermore, for every $s,t\geq0$, expanding \eqref{eq:takeradproj} in the Walsh basis, we see that
\begin{equation} \label{eq:twovarcontract1}
g_{s,t}^\ast(\e,\delta) = \sum_{B\subseteq\{1,\ldots,n\}} e^{-t|B|} (1-e^{-t})^{n-|B|} g_s^\ast\Big(\sum_{i\in B} \e_i e_i +\sum_{i\in\{1,\ldots,n\}\setminus B} \delta_i e_i\Big),
\end{equation}
where $\{e_i\}_{i=1}^n$ is the standard basis of $\R^n$. Therefore, for every $s, t\geq0$, we have
\begin{equation} \label{eq:twovarcontract2}
\|g_{s,t}^\ast\|_{L_q(\{-1,1\}^n\times\{-1,1\}^n;X^\ast)} \leq \sum_{k=0}^n \binom{n}{k} e^{-tk}(1-e^{-t})^{n-k} \|g_s^\ast\|_{L_q(\{-1,1\}^n;X^\ast)} =1.
\end{equation}
Combining \eqref{eq:writenormfun}, \eqref{eq:usenormfun}, \eqref{eq:takeradproj} and \eqref{eq:twovarcontract2} with H\"older's inequality, we deduce that
\begin{equation} \label{eq:loong}
\begin{split}
\|e^{-s\Delta}f\|_{L_p(\{-1,1\}^n;X)} & \stackrel{\eqref{eq:usenormfun}\wedge\eqref{eq:takeradproj}}{=} \frac{1}{4^n} \sum_{\e,\delta\in\{-1,1\}^n} \int_s^\infty \frac{1}{e^t-1} \langle g_{s,t}^\ast(\e,\delta), \sum_{i=1}^n \delta_i \partial_i f(\e)\rangle \diff t
\\ & \stackrel{\eqref{eq:twovarcontract2}}{\leq} \Big( \int_s^\infty \frac{1}{e^t-1} \diff t\Big) \Big( \frac{1}{2^n}\sum_{\delta\in\{-1,1\}^n} \Big\| \sum_{i=1}^n \delta_i \partial_i f\Big\|^p_{L_p(\{-1,1\}^n;X)}\Big)^{1/p}
\\ & = \log\Big( \frac{e^s}{e^s-1}\Big) \Big( \frac{1}{2^n}\sum_{\delta\in\{-1,1\}^n} \Big\| \sum_{i=1}^n \delta_i \partial_i f\Big\|^p_{L_p(\{-1,1\}^n;X)}\Big)^{1/p}.
\end{split}
\end{equation}
Furthermore, since $f$ has degree at most $d$, Theorem \ref{thm:reverseheatgeneral} implies that for $s\geq0$,
\begin{equation} \label{eq:usereverse}
\frac{1}{T_d(e^s)} \|f\|_{L_p(\{-1,1\}^n;X)} \leq \|e^{-s\Delta}f\|_{L_p(\{-1,1\}^n;X)}.
\end{equation}
Combining \eqref{eq:loong} and \eqref{eq:usereverse}, we conclude that for every $s\geq0$,
\begin{equation} \label{eq:foreverys}
\|f\|_{L_p(\{-1,1\}^n;X)} \leq T_d(e^s)\log\Big( \frac{e^s}{e^s-1}\Big) \Big( \frac{1}{2^n}\sum_{\delta\in\{-1,1\}^n} \Big\| \sum_{i=1}^n \delta_i \partial_i f\Big\|^p_{L_p(\{-1,1\}^n;X)}\Big)^{1/p}.
\end{equation}
As in Remark \ref{rem:ole}, using that $T_d(e^s)\leq e^{d^2s}$ for $s\geq0$ and a straightforward optimization we get
\begin{equation} \label{eq:minins}
\min_{s\geq0} T_d(e^s)\log\Big(\frac{e^s}{e^s-1}\Big) \leq \min_{s\geq0} e^{d^2s} \log\Big(\frac{e^s}{e^s-1}\Big) \leq 3(\log d+1).
\end{equation}
Plugging \eqref{eq:minins} in \eqref{eq:foreverys}, we finally deduce Pisier's inequality \eqref{eq:Pisierlowfreq}.
\hfill$\Box$

\medskip

We now prove the Bernstein--Markov type inequality \eqref{eq:laplaciangeneral} for the hypercube Laplacian. In the proof, we will need Markov's inequality (see \cite[Theorem~5.1.8]{BE95}) which asserts that
\begin{equation} \label{eq:markov}
\max \{ |p'(1)|: \ \mathrm{deg}(p)\leq d \ \mbox{and} \ \|p\|_{\ms{C}([-1,1])}=1\big\} = d^2,
\end{equation}
where the equality is achieved for the $d$-th Chebyshev polynomial of the first kind $T_d(x)$. Our proof is similar to the one presented in \cite{FHKL16}, \mbox{which also crucially relied on Markov's inequality \eqref{eq:markov}.}

\medskip

\noindent {\it Proof of Theorem \ref{thm:laplaciangeneral}.}
Fix $n\in\N$, $d\in\{1,\ldots,n\}$, $p\in[1,\infty]$ and let $f:\{-1,1\}^n\to X$ be a function of degree at most $d$. Consider the linear functional \mbox{$\psi:\big(\mathrm{span}\{1,x,\ldots,x^d\}, \|\cdot\|_{\ms{C}([-1,1])}\big) \to\C$ given by}
\begin{equation}
\psi\Big( \sum_{k=0}^d a_k x^k\Big) \eqdef \sum_{k=0}^d ka_k,
\end{equation}
or $\psi(p)=p'(1)$ when $p$ is a polynomial with $\mathrm{deg}(p)\leq d$. Then, Markov's inequality \eqref{eq:markov} can be rewritten as
\begin{equation} \label{ineq6}
\forall \ p\in\mathrm{span}\{1,x,\ldots,x^d\}, \ \ \ |\psi(p)| \leq d^2 \|p\|_{\ms{C}([-1,1])}.
\end{equation}
Therefore, by the Hahn--Banach theorem and the Riesz representation theorem, there exists a complex measure $\nu$ on $[-1,1]$ such that $\|\nu\|_{\ms{M}([-1,1])}\leq d^2$ and for every polynomial $p$, we have
\begin{equation} \label{ide6}
\mathrm{deg}(p)\leq d \ \ \ \Longrightarrow \ \ \ \int_{-1}^1 p(x)\diff\nu(x) = p'(1).
\end{equation}
Since $\widehat{f}(A)=0$ when $|A|>d$, we get that
\begin{equation}
\begin{split}
\|\Delta f&\|_{L_p(\{-1,1\}^n;X)}  \stackrel{\eqref{ide6}}{=} \Big\| \int_{-1}^1 x^\Delta f \diff\nu(x)\Big\|_{L_p(\{-1,1\}^n;X)} \leq \int_{-1}^1 \|x^\Delta f\|_{L_p(\{-1,1\}^n;X)} \diff|\nu|(x) 
\\ & \stackrel{\eqref{eq:contract<0}}{\leq} \int_{-1}^1 \|f\|_{L_p(\{-1,1\}^n;X)} \diff|\nu|(x) = \|\nu\|_{\ms{M}([-1,1])} \|f\|_{L_p(\{-1,1\}^n;X)} \stackrel{\eqref{ineq6}}{\leq} d^2 \|f\|_{L_p(\{-1,1\}^n;X)},
\end{split}
\end{equation}
which concludes the proof of the theorem.
\hfill$\Box$

\medskip

\begin{remark} \label{rem:timederivative}
One can generalize Theorem \ref{thm:laplaciangeneral} by showing that for every $k\in\N$, we have
\begin{equation} \label{eq:timederivative}
\|\Delta(\Delta-1)\cdots(\Delta-k+1) f\|_{L_p(\{-1,1\}^n;X)} \leq \frac{d^2\big(d^2-1^2\big)\cdots\big(d^2-(k-1)^2\big)}{1\cdot 3\cdots (2k-1)} \|f\|_{L_p(\{-1,1\}^n;X)}.
\end{equation}
Theorem \ref{thm:reverseheatgeneral} corresponds to the case $k=1$. To derive \eqref{eq:timederivative}, one can use the same duality argument used in the proofs of Theorems \ref{thm:reverseheatgeneral} and \ref{thm:laplaciangeneral} for the linear functional $\psi_k(p) = p^{(k)}(1)$ along with Markov's inequality for higher derivatives (see \cite[Theorem~5.2.1]{BE95}), which asserts that
\begin{equation} \label{eq:vamarkov}
\max \{ |p^{(k)}(1)|: \ \mathrm{deg}(p)\leq d \ \mbox{and} \ \|p\|_{\ms{C}([-1,1])}=1\big\} = T_d^{(k)}(1) = \frac{d^2\big(d^2-1^2\big)\cdots\big(d^2-(k-1)^2\big)}{1\cdot 3\cdots (2k-1)}.
\end{equation}
\end{remark}

We conclude this section by proving Theorem \ref{thm:cotype}.

\medskip

\noindent {\it Proof of Theorem \ref{thm:cotype}.} Assume that $(X,\|\cdot\|_X)$ does not have finite cotype. By the Maurey--Pisier theorem \cite{MP76}, for every $\theta>0$ and $m\in\N$ there exists a linear operator $\msf{S}_m:\ell_\infty^m\to X$ such that
\begin{equation} \label{eq:usemp}
\forall \ y\in\ell_\infty^m, \ \ \ \|y\|_\infty \leq \|\msf{S}_my\|_X \leq (1+\theta)\|y\|_\infty.
\end{equation}
Let $n\in\N$ with $d\in\{1,\ldots,n\}$. Define the function $f:\{-1,1\}^n \to L_\infty(\{-1,1\}^n;\R)$ by
\begin{equation}
\big[ f(\e) \big] (\delta) \eqdef T_d\Big(\frac{\e_1\delta_1+\cdots\e_n\delta_n}{n}\Big),
\end{equation}
where $T_d(x)$ is the $d$-th Chebyshev polynomial of the first kind and consider the composition $F: \{-1,1\}^n\to X$ given by $F=\msf{S}_{2^n}\circ f$, where $L_\infty(\{-1,1\}^n;\R)$ is naturally identified with $\ell_\infty^{2^n}$. Since $T_d(x)$ is a polynomial of degree $d$, the functions $f$ and $F$ are also of degree at most $d$ and
\begin{equation}
\|F(\e)\|_X \stackrel{\eqref{eq:usemp}}{\leq}(1+\theta) \|f(\e)\|_{L_\infty(\{-1,1\}^n;\R)} = (1+\theta) \max_{\delta\in\{-1,1\}^n} T_d\Big(\frac{\e_1\delta_1+\cdots\e_n\delta_n}{n}\Big) = 1+\theta,
\end{equation}
for every $\e\in\{-1,1\}^n$, because $\|T_d\|_{\ms{C}([-1,1])} = T_d(1)=1$. Therefore, we also have 
\begin{equation} \label{eq:Fformp}
\|F\|_{L_p(\{-1,1\}^n;X)} = \Big( \frac{1}{2^n} \sum_{\e\in\{-1,1\}^n} \|F(\e)\|_X^p\Big)^{1/p} \leq 1+\theta.
\end{equation}
Furthermore, for every $\e,\delta\in\{-1,1\}^n$,
\begin{equation}
\big[\partial_i f(\e)\big](\delta) = \frac{1}{2}\Big(T_d\Big(\frac{\e_1\delta_1+\cdots+ \e_n\delta_n}{n}\Big) - T_d\Big(\frac{\e_1\delta_1+\cdots+\e_n\delta_n}{n} - \frac{2\e_i\delta_i}{n}\Big)\Big),
\end{equation}
which implies the identity
\begin{equation}
\big[\Delta f(\e)\big](\delta) = \frac{1}{2}\Big(n T_d\Big(\frac{\e_1\delta_1+\cdots+\e_n\delta_n}{n}\Big) - \sum_{i=1}^n T_d\Big(\frac{\e_1\delta_1+\cdots+\e_n\delta_n}{n} - \frac{2\e_i\delta_i}{n}\Big)\Big).
\end{equation}
Therefore, taking $\e=\delta$, we deduce that
\begin{equation} \label{eq:cosinecompute}
\begin{split}
\|\Delta f(\e)\|_{L_\infty(\{-1,1\}^n;\R)}& \geq \frac{n}{2} \Big( T_d(1) - T_d\Big(1-\frac{2}{n}\Big)\Big) = \frac{n}{2} \Big( 1 - T_d\Big(1-\frac{2}{n}\Big)\Big),
\end{split}
\end{equation}
from which we derive the estimate
\begin{equation} \label{eq:deltaformp}
\|\Delta F\|_{L_p(\{-1,1\}^n;X)} \stackrel{\eqref{eq:usemp}}{\geq} \|\Delta f\|_{L_p(\{-1,1\}^n;L_\infty(\{-1,1\}^n;\R))} \stackrel{\eqref{eq:cosinecompute}}{\geq}  \frac{n}{2}\Big( 1-T_d\Big(1-\frac{2}{n}\Big)\Big).
\end{equation}
Applying \eqref{eq:cotypeassumption} for the function $F:\{-1,1\}^n\to X$ and using \eqref{eq:Fformp} and \eqref{eq:deltaformp}, we finally get
\begin{equation}
\begin{split}
\frac{n}{2} \Big( 1 - T_d\Big(1-\frac{2}{n}\Big)\Big) &
\stackrel{\eqref{eq:deltaformp}}{\leq} \|\Delta F\|_{L_p(\{-1,1\}^n;X)} \\ & \stackrel{\eqref{eq:cotypeassumption}}{\leq} (1-\eta) d^2 \|F\|_{L_p(\{-1,1\}^n;X)}  \stackrel{\eqref{eq:Fformp}}{\leq} (1-\eta)(1+\theta)d^2.
\end{split}
\end{equation}
Letting $n\to\infty$, we get
\begin{equation} \label{eq:almostcontr}
d^2=T_d'(1) \leq (1-\eta)(1+\theta)d^2.
\end{equation}
Finally, letting $\theta\to0^+$, \eqref{eq:almostcontr} becomes $d^2\leq (1-\eta)d^2$ which is a contradiction.
\hfill$\Box$

\begin{remark} \label{rem:sharpness}
Considering the function $f:\{-1,1\}^n\to\C$ given by $f(\e) = T_d\big(\frac{\e_1+\cdots+\e_n}{n}\big)$ for $n\to\infty$ shows that Theorem \ref{thm:laplaciangeneral} is also sharp when $X=\C$ and $p=\infty$.
\end{remark}


\section{Estimates for $K$-convex Banach spaces} \label{sec:3}

The main vector valued Fourier analytic tool which we will exploit in this section is the following fact which lies at the heart of the proof of Pisier's $K$-convexity theorem \cite{Pis82}.

\begin{theorem}  [Pisier] \label{thm:pisKconv}
A Banach space $(X,\|\cdot\|_X)$ is $K$-convex if and only if for every $p\in(1,\infty)$ there exist $\theta=\theta(p,X)\in\big(0,\frac{\pi}{2}\big]$ and $M=M(p,X)\in[1,\infty)$ such that for every $n\in\N$
\begin{equation} \label{eq:conclusionKconv}
|\arg z|\leq \theta \ \ \ \Longrightarrow \ \ \ \|e^{-z\Delta}\|_{{L_p(\{-1,1\}^n;X)}\to {L_p(\{-1,1\}^n;X)}} \leq M,
\end{equation}
where $\|T\|_{{L_p(\{-1,1\}^n;X)}\to {L_p(\{-1,1\}^n;X)}}$ is the operator norm of\mbox{ $T$ from ${L_p(\{-1,1\}^n;X)}$ to itself.}
\end{theorem}

The fact that a Banach space satisfying \eqref{eq:conclusionKconv} for some $\theta\in\big(0,\frac{\pi}{2}\big]$ and $M\in[1,\infty)$ is $K$-convex is simple. Indeed, let $a = \frac{\pi}{\tan\theta}$ so that every point $z$ in the interval with endpoints $a+i\pi$, $a-i\pi$ has $|\arg z|\leq \theta$. Then, for $k\in\{0,1,\ldots,n\}$,
\begin{equation}
\frac{1}{2\pi}\int_{-\pi}^\pi e^{ikt} e^{-(a+it)\Delta} \diff t = \frac{1}{2\pi} \int_{-\pi}^\pi \sum_{j=0}^n e^{ikt - (a+it)j} \msf{Rad}_j \diff t = e^{-ka} \msf{Rad}_k,
\end{equation}
which implies that
\begin{equation} \label{eq:boundRadk}
\begin{split}
\big\|\msf{Rad}_k\big\|&_{{L_p(\{-1,1\}^n;X)}\to {L_p(\{-1,1\}^n;X)}}  = e^{ka} \Big\|\frac{1}{2\pi} \int_{-\pi}^\pi e^{ikt} e^{-(a+it)\Delta} \diff t\Big\|_{{L_p(\{-1,1\}^n;X)}\to {L_p(\{-1,1\}^n;X)}} 
\\ &\leq \frac{e^{ka}}{2\pi} \int_{-\pi}^\pi \|e^{-(a+it)\Delta}\|_{{L_p(\{-1,1\}^n;X)}\to {L_p(\{-1,1\}^n;X)}} \diff t \stackrel{\eqref{eq:conclusionKconv}}{\leq} M e^{ka}.
\end{split}
\end{equation}
In particular, $\sup_{n\in\N} \|\msf{Rad}\|_{{L_p(\{-1,1\}^n;X)}\to {L_p(\{-1,1\}^n;X)}} \leq Me^a <\infty$, i.e. $(X,\|\cdot\|_X)$ is $K$-convex.

\smallskip

Here and throughout, we will denote by $\mb{D} = \{z\in\C: |z|<1\}$ the open unit disc. For $r\in[1,\infty)$ consider the lens domain 
\begin{equation} \label{eq:Omega(r)def}
\Omega(r) \eqdef \big\{z\in \C: \ \max\big\{\big|z-i\sqrt{r^2-1}\big|, \big|z+i\sqrt{r^2-1}\big|\big\} < r \big\} \subseteq\mb{D} 
\end{equation}
and notice that $\partial\Omega(r)$ is a Jordan curve with an interior angle $\theta(r)=2\arcsin(1/r)$ at $z=\pm1$. We will need the following simple consequence of Theorem \ref{thm:pisKconv}.

\begin{corollary} \label{cor:lens}
Let $(X,\|\cdot\|_X)$ be a $K$-convex Banach space and $p\in(1,\infty)$. Then there exist some $r=r(p,X)\in[1,\infty)$ and $K=K(p,X)\in[1,\infty)$ such that
\begin{equation} \label{eq:LemmaOmega(r)}
\sup_{w\in\Omega(r)} \sup_{n\in\N} \|w^{\Delta}\|_{{L_p(\{-1,1\}^n;X)}\to {L_p(\{-1,1\}^n;X)}} \leq K.
\end{equation}
\end{corollary}

\begin{proof}
By Theorem \ref{thm:pisKconv} and \eqref{eq:symmetrysemigroup}, for every $w\in\{e^{-z}: \ |\mathrm{arg}z|<\theta\}$, we have
\begin{equation} \label{eq:boundsymmetricdomain}
\|w^{\Delta}\|_{{L_p(\{-1,1\}^n;X)}\to {L_p(\{-1,1\}^n;X)}} = \|(-w)^{\Delta}\|_{{L_p(\{-1,1\}^n;X)}\to {L_p(\{-1,1\}^n;X)}} \leq M.
\end{equation}
Moreover, by \eqref{eq:boundRadk}, we get
\begin{equation}
\begin{split}
\|w^{\Delta}\|&_{{L_p(\{-1,1\}^n;X)}\to {L_p(\{-1,1\}^n;X)}} \\ & \leq \sum_{k=0}^n |w|^k \big\|\msf{Rad}_k\big\|_{{L_p(\{-1,1\}^n;X)}\to {L_p(\{-1,1\}^n;X)}}
 \stackrel{\eqref{eq:boundRadk}}{\leq} M \sum_{k=0}^n \big(|w|e^a\big)^k,
\end{split}
\end{equation}
which implies that for $|w|<e^{-2a}$, we also have the estimate
\begin{equation} \label{eq:bounddisc}
\|w^{\Delta}\|_{{L_p(\{-1,1\}^n;X)}\to {L_p(\{-1,1\}^n;X)}} \leq \frac{M}{1-e^{-a}}.
\end{equation}
Combining the domains for which \eqref{eq:boundsymmetricdomain} and \eqref{eq:bounddisc} hold, one can easily deduce that there exists $r\in[1,\infty)$ such that
\begin{equation}
\Omega(r) \subseteq \{e^{-z}: \ |\mathrm{arg}z|<\theta\} \cup \{-e^{-z}: \ |\mathrm{arg}z|<\theta\} \cup\{w\in\mb{D}: \ |w|<e^{-2a}\}
\end{equation}
and \eqref{eq:LemmaOmega(r)} follows with $K=\frac{M}{1-e^{-a}}$.
\end{proof}

We record for ease of future reference the following \mbox{calculation of conformal mappings.}

\begin{lemma} \label{lem:computeconformal}
For $\alpha\in(0,\infty)$, let $\phi_\alpha:\mb{C}\setminus(-\infty,0]\to\mb{C}$ be a holomorphic branch of $z\mapsto z^{\alpha}$. Also, consider the M\"obius transformations $\psi_1(z) = \frac{1-z}{1+z}$ and $\psi_2(z)=\frac{z+1}{z-1}$. Then, for every $r\in[1,\infty)$, the map $\psi_1\circ \phi_{\pi/\theta(r)}\circ\psi_1$ is a conformal equivalence between $\Omega(r)$ and the unit disc $\mb{D}$. Furthermore, the map $\psi_2\circ\phi_{\pi/(2\pi-\theta(r))}\circ\psi_2$ is a conformal equivalence between the complement $\overline{\Omega(r)}^c$ of $\overline{\Omega(r)}$ and the complement $\overline{\mb{D}}^c$ of the closed unit disc.
\end{lemma}

\begin{proof}
Since $\psi_1(-1)=0$, $\psi_1(1)=\infty$ and M\"obius transformations preserve angles, we have
\begin{equation}
\psi_1\big(\Omega(r)\big) = \Big\{ z\in\mb{C}: \ |\arg z|< \frac{\theta(r)}{2}\Big\}.
\end{equation}
Therefore, composing with $\phi_{\pi/\theta(r)}$, we get
\begin{equation}
\phi_{\pi/\theta(r)}\circ\psi_1\big(\Omega(r)\big) = \big\{z\in\mb{C}: \ \mathrm{Re} z>0\big\},
\end{equation}
which immediately implies that $\psi_1\circ \phi_{\pi/\theta(r)}\circ\psi_1\big(\Omega(r)\big)=\mb{D}$. Since all the functions involved are conformal equivalences at their domains of definition the proof is complete. The proof of the claim for the complement $\overline{\Omega(r)}^c$ of $\overline{\Omega(r)}$ is identical.
\end{proof}

We can now proceed with the proof of Theorem \ref{thm:mendel-naor}. The crucial proposition is the following.

\begin{proposition} \label{prop:crucialheat}
Fix a $K$-convex Banach space $(X,\|\cdot\|_X)$, $p\in(1,\infty)$ and $K\in(0,\infty)$. Let $\Omega\subseteq\mb{D}$ be a simply connected domain bounded by a Jordan curve $\gamma$ such that $(-1,1) \subseteq \Omega$ and suppose that
\begin{equation} \label{eq:assumppropKconv}
\sup_{w\in \Omega} \sup_{n\in\N} \|w^\Delta\|_{{L_p(\{-1,1\}^n;X)}\to {L_p(\{-1,1\}^n;X)}} \leq K.
\end{equation}
Also, let $\phi_\Omega: \Omega\to \mb{D}$ be a conformal mapping of $\Omega$ onto $\mb{D}$ with $\phi_\Omega(0)=0$ which extends continuously as a homeomorphism between $\Omega\cup\gamma$ and $\overline{\mb{D}}$. Then, for every $n,d\in\N$ with $d\in\{0,1,\ldots,n-1\}$ and every function $f:\{-1,1\}^n\to X$ in the $d$-th tail space, we have
\begin{equation}
\forall \ t\geq0, \ \ \ \|e^{-t\Delta}f\|_{L_p(\{-1,1\}^n;X)} \leq K \big|\phi_\Omega\big(e^{-t}\big)\big|^d \|f\|_{L_p(\{-1,1\}^n;X)}.
\end{equation}
\end{proposition}

\begin{proof}
Fix $n\in\N$, $d\in\{0,1,\ldots,n-1\}$ and let $f:\{-1,1\}^n\to X$ be a function in the $d$-th tail space. For $t\geq0$, consider the linear functional \mbox{$\zeta_t:\big(\mathrm{span}\{w^d,w^{d+1},\ldots,w^n\}, \|\cdot\|_{\ms{C}(\overline{\Omega})}\big) \to\C$ given by}
\begin{equation}
\zeta_t\Big(\sum_{k=d}^n a_k w^k \Big) \eqdef \sum_{k=d}^n a_k e^{-tk},
\end{equation} 
or $\zeta_t(p) = p(e^{-t})$ if $p(w)$ is a polynomial of degree $n$ which is a multiple of $w^d$. If $p$ is such a polynomial, consider the function $h:\Omega\to\C$ given by
\begin{equation}
\forall \ w\in\Omega, \ \ \ h(w) \eqdef \frac{p(w)}{\phi_\Omega(w)^d}.
\end{equation}
Notice that $\phi_\Omega$ does not vanish on $\Omega\setminus\{0\}$ and furthermore it has a single root at 0, therefore the multiplicity of the root 0  in the numerator is at least the multiplicity in the denominator. Thus $h$ is a holomorphic function on $\Omega$. Furthermore, for $w\in\gamma$, we have $|\phi_\Omega(w)| = 1$ since $\phi_\Omega$ is a homeomorphism between $\gamma$ and $\partial\mb{D}$ (such a conformal map $\phi_\Omega$ always exists by Caratheodory's theorem, see \cite[Theorem~5.1.1]{Kra06}). Therefore, by the maximum principle,
\begin{equation}
\sup_{w\in\Omega} |h(w)| = \max_{w\in\gamma} |h(w)| = \max_{w\in\gamma} \frac{|p(w)|}{|\phi_\Omega(w)|^d} = \|p\|_{\ms{C}(\gamma)}= \|p\|_{\ms{C}(\overline{\Omega})},
\end{equation}
which implies that for $w\in\overline{\Omega}$, we have
\begin{equation} \label{eq:usedconformal}
|p(w)| \leq |\phi_\Omega(w)|^d \|p\|_{\ms{C}(\overline{\Omega})}.
\end{equation}
Applying \eqref{eq:usedconformal} for $w=e^{-t}\in\Omega$, we deduce that the linear functional $\zeta_t$ satisfies
\begin{equation} \label{ineq1}
\forall \ p\in\mathrm{span}\{w^d,w^{d+1},\ldots,w^n\}, \ \ \ |\zeta_t(p)| \leq \big|\phi_\Omega\big(e^{-t}\big)\big|^d \|p\|_{\ms{C}(\overline{\Omega})}.
\end{equation}
Therefore, by the Hahn--Banach theorem and the Riesz representation theorem, there exists a complex measure $\tau_t$ on $\overline{\Omega}$ such that $\|\tau_t\|_{\ms{M}(\overline{\Omega})} \leq |\phi_\Omega(e^{-t})|^d$ and
\begin{equation} \label{ide1}
p\in\mathrm{span}\{w^d,w^{d+1},\ldots,w^n\} \ \ \ \Longrightarrow \ \ \ \int_{\overline{\Omega}} p(w)\diff\tau_t(w) = p\big(e^{-t}\big).
\end{equation}
Since $\widehat{f}(A)=0$ when $|A|<d$, we get that
\begin{equation}
\begin{split}
\|e^{-t\Delta}f\|_{L_p(\{-1,1\}^n;X)}& \stackrel{\eqref{ide1}}{=} \Big\| \int_{\overline{\Omega}} w^\Delta f \diff\tau_t(w)\Big\|_{L_p(\{-1,1\}^n;X)} \leq \int_{\overline{\Omega}} \|w^\Delta f\|_{L_p(\{-1,1\}^n;X)} \diff|\tau_t|(w)
\\ & \stackrel{\eqref{eq:assumppropKconv}}{\leq} K \int_{\overline{\Omega}} \|f\|_{L_p(\{-1,1\}^n;X)} \diff|\tau_t|(w) = K\|\tau_t\|_{\ms{M}(\overline{\Omega})} \|f\|_{L_p(\{-1,1\}^n;X)}
\\ & \stackrel{\eqref{ineq1}}{\leq} K\big| \phi_\Omega\big(e^{-t}\big)\big|^d \|f\|_{L_p(\{-1,1\}^n;X)},
\end{split}
\end{equation}
which concludes the proof.
\end{proof}

\noindent {\it Proof of Theorem \ref{thm:mendel-naor}.} Fix $n\in\N$, $d\in\{0,1,\ldots,n-1\}$, $p\in(1,\infty)$ and let $f:\{-1,1\}^n\to X$ be a function in the $d$-th tail space. Using the notation introduced after Theorem \ref{thm:pisKconv}, for every $t\geq 2a$,
\begin{equation}
\begin{split}
\|e^{-t\Delta}f\|_{L_p(\{-1,1\}^n;X)} \leq \sum_{k=d}^n e^{-tk} \big\|\msf{Rad}_kf\big\|_{L_p(\{-1,1\}^n;X)}
 \stackrel{\eqref{eq:boundRadk}}{\leq} Ke^{-\frac{1}{2}td} \|f\|_{L_p(\{-1,1\}^n;X)},
\end{split}
\end{equation}
where $K=\frac{M}{1-e^{-a}}$. We will now treat the range $t\leq 2a$. By Corollary \ref{cor:lens}, there exists $r=r(p,X)\in[1,\infty)$ such that
\begin{equation}
\sup_{w\in\Omega(r)} \sup_{n\in\N} \|w^\Delta\|_{{L_p(\{-1,1\}^n;X)}\to {L_p(\{-1,1\}^n;X)}} \leq K,
\end{equation}
where $\Omega(r)$ is the domain \eqref{eq:Omega(r)def}. Therefore, by Proposition \ref{prop:crucialheat}, we have
\begin{equation}
\forall \ t\geq0, \ \ \ \|e^{-t\Delta}f\|_{L_p(\{-1,1\}^n;X)} \leq K \big|\phi_{\Omega(r)}\big(e^{-t}\big)\big|^d \|f\|_{L_p(\{-1,1\}^n;X)}.
\end{equation}
To conclude the proof, we will show that there exists $c=c(p,X)$ such that
\begin{equation} \label{eq:boundconformalbypower}
\forall \ 0\leq t\leq 2a, \ \ \ \big|\phi_{\Omega(r)}\big(e^{-t}\big)\big| \leq e^{-c t^{\pi/\theta(r)}},
\end{equation}
where $\theta(r) = 2\arcsin(1/r) \in (0,\pi]$ which would then imply the conclusion of Theorem \ref{thm:mendel-naor} with $A(p,X) = \frac{\pi}{\theta(r(p,X))} \in[1,\infty)$. By Lemma \ref{lem:computeconformal}, if $\theta=\theta(r)$, we have
\begin{equation} \label{eq:phiomegar}
\phi_{\Omega(r)}\big(e^{-t}\big) = \frac{(1+e^{-t})^{\pi/\theta}- (1-e^{-t})^{\pi/\theta}}{(1+e^{-t})^{\pi/\theta} + (1-e^{-t})^{\pi/\theta}}.
\end{equation}
Therefore, 
\begin{equation}
\begin{split}
\log\big|\phi_{\Omega(r)}\big(e^{-t}\big)\big| = \log\left( 1- \frac{2(1-e^{-t})^{\pi/\theta}}{(1+e^{-t})^{\pi/\theta} + (1-e^{-t})^{\pi/\theta}}\right) & \leq - \frac{2(1-e^{-t})^{\pi/\theta}}{(1+e^{-t})^{\pi/\theta} + (1-e^{-t})^{\pi/\theta}}
\end{split}
\end{equation}
and for $t\leq2a$,
\begin{equation}
 - \frac{2(1-e^{-t})^{\pi/\theta}}{(1+e^{-t})^{\pi/\theta} + (1-e^{-t})^{\pi/\theta}} \leq -2^{1-\frac{\pi}{\theta}}(1-e^{-t})^{\pi/\theta} \leq \frac{-2^{1-\frac{\pi}{\theta}}(1-e^{-2a})^{\pi/\theta}}{(2a)^{\pi/\theta}} t^{\pi/\theta},
\end{equation}
which implies \eqref{eq:boundconformalbypower} with $c = \frac{2^{1-\frac{\pi}{\theta}}(1-e^{-2a})^{\pi/\theta}}{(2a)^{\pi/\theta}}$.
\hfill$\Box$

\medskip

We now proceed to prove the dual of Theorem \ref{thm:mendel-naor}, namely the improved lower bound on the decay of the heat semigroup, Theorem \ref{thm:reverseheatKconvex}. Even though Theorems \ref{thm:mendel-naor} and \ref{thm:reverseheatKconvex} are not formally dual to one another, there is a continuous analogy between the techniques used in their proofs. We will need the following elementary result from complex analysis.

\begin{lemma}
Let $\Omega\subseteq\C$ be a bounded simply connected domain and consider a conformal mapping $\phi_{\Omega^c}:\overline{\Omega}^c\to\overline{\mb{D}}^c$ of the complement of $\overline{\Omega}$ onto the complement of the closed unit disc with $\phi_{\Omega^c}(\infty)=\infty$. Then,
\begin{equation} \label{eq:limitquotient}
\lim_{z\to\infty} \frac{\phi_{\Omega^c}(z)}{z} = \beta \in \mb{C}\setminus\{0\}.
\end{equation}
\end{lemma}

\begin{proof}
Without loss of generality assume that $0\in\Omega$ and let $V \eqdef\big\{z\in\mb{C}: \ \frac{1}{z}\in \overline{\Omega}^c\big\}$. Then $V\cup\{0\}$ is a domain and the function $F:V\to \C$ given by
\begin{equation}
\forall \ z\in V, \ \ \ F(z)\eqdef \frac{1}{\phi_{\Omega^c}\big(1/z\big)}
\end{equation}
is holomorphic, injective and satisfies $|F(z)|\leq 1$ for every $z\in V$. Therefore, by Riemann's theorem on removable singularities, $F$ can be holomorphically extended at 0. Furthermore, $F(0)=0$, since otherwise $\lim_{z\to\infty} \phi_{\Omega^c}(z) \in \mb{C}$. Since $F$ is injective, we also have $F'(0)\neq0$. Thus,
\begin{equation}
\lim_{z\to\infty} \frac{\phi_{\Omega^c}(z)}{z} = \lim_{w\to0} w \phi_{\Omega^c}(1/w) = \lim_{w\to0} \frac{w}{F(w)} = \frac{1}{F'(0)} \in \mb{C}\setminus\{0\},
\end{equation}
which concludes the proof.
\end{proof}

\begin{proposition} \label{prop:reverseheatcrucial}
Fix a $K$-convex Banach space $(X,\|\cdot\|_X)$, $p\in(1,\infty)$ and $K\in(0,\infty)$. Let $\Omega\subseteq\mb{D}$ be a simply connected domain bounded by a Jordan curve $\gamma$ and suppose that
\begin{equation} \label{eq:assumppropKconv0}
\sup_{w\in\Omega} \sup_{n\in\N}\|w^\Delta\|_{{L_p(\{-1,1\}^n;X)}\to {L_p(\{-1,1\}^n;X)}} \leq K.
\end{equation}
Also, let $\phi_{\Omega^c}:\overline{\Omega}^c\to\overline{\mb{D}}^c$ be a conformal mapping of the complement of $\overline{\Omega}$ onto the complement of the closed unit disc which extends continuously as a homeomorphism between $\Omega^c$ and $\mb{D}^c$. Then, for every $n,d\in\N$ with $d\in\{1,\ldots,n\}$ and every function\mbox{ $f:\{-1,1\}^n\to X$ of degree at most $d$, we have}
\begin{equation} \label{eq:reverseheatconformal}
\forall \ t\geq0, \ \ \ \|e^{-t\Delta}f\|_{L_p(\{-1,1\}^n;X)} \geq \frac{1}{K\big|\phi_{\Omega^c}\big(e^{t}\big)\big|^d} \|f\|_{L_p(\{-1,1\}^n;X)}.
\end{equation}
\end{proposition}

\begin{proof}
Fix $n\in\N$, $d\in\{1,\ldots,n\}$ and let $f:\{-1,1\}^n\to X$ be a function of degree at most $d$. For $t\geq0$, consider the linear functional $\varphi_t : \big(\mathrm{span}\{1,w,\ldots,w^d\},\|\cdot\|_{\ms{C}(\overline{\Omega})}\big)\to \C$ given by
\begin{equation}
\varphi_t\Big(\sum_{k=0}^d a_k w^k\Big) \eqdef \sum_{k=0}^d a_k e^{tk},
\end{equation}
or $\varphi_t(p) = p(e^t)$ when $p$ is a polynomial with $\mathrm{deg}(p)\leq d$. If $p$ is such a polynomial, consider the function $h:\overline{\Omega}^c\to \mb{C}$ given by
\begin{equation}
\forall \ z\in\overline{\Omega}^c, \ \ \ h(z) \eqdef \frac{p(z)}{\phi_{\Omega^c}(z)^d}.
\end{equation}
Notice that, since the conformal map $\phi_{\Omega^c}$ extends as a homeomorphism between $\Omega^c$ and $\mb{D}^c$, it also satisfies $\phi_{\Omega^c}(\infty)=\infty$. Thus, \eqref{eq:limitquotient} implies that $\lim_{z\to\infty} h(z)\in\mb{C}$ and therefore $h$ is a holomorphic function on $\overline{\Omega}^c\cup\{\infty\}$. Furthermore, for $z\in\gamma$, we have $|\phi_{\Omega^c}(z)| = 1$ since $\phi_{\Omega^c}$ is a homeomorphism between $\gamma$ and $\partial\mb{D}$. Therefore, by the maximum principle,
\begin{equation}
\sup_{z\in\Omega^c} |h(z)| = \max_{z\in\gamma} |h(z)| = \max_{z\in\gamma}\frac{|p(z)|}{|\phi_{\Omega^c}(z)|^d} = \|p\|_{\ms{C}(\gamma)} = \|p\|_{\ms{C}(\overline{\Omega})},
\end{equation}
which implies that for $z\in\Omega^c$, we have
\begin{equation} \label{eq:usedconformalout}
|p(z)| \leq |\phi_{\Omega^c}(z)|^d \|p\|_{\ms{C}(\overline{\Omega})}.
\end{equation}
Applying \eqref{eq:usedconformalout} for $z=e^t \in\Omega^c$, we deduce that the linear functional $\varphi_t$ satisfies
\begin{equation} \label{ineq2}
\forall \ p\in\mathrm{span}\{1,w,\ldots,w^d\}, \ \ \ |\varphi_t(p)| \leq \big|\phi_{\Omega^c}(e^t)\big|^d \|p\|_{\ms{C}(\overline{\Omega})}.
\end{equation}
Therefore, by the Hahn--Banach theorem and the Riesz representation theorem, there exists a complex measure $\rho_t$ on $\overline{\Omega}$ such that $\|\rho_t\|_{\ms{M}(\overline{\Omega})}\leq |\phi_{\Omega^c}(e^t)|^d$ and
\begin{equation} \label{ide2}
p\in\mathrm{span}\{1,w,\ldots,w^d\} \ \ \ \Longrightarrow \ \ \ \int_{\overline{\Omega}} p(w)\diff\rho_t(w) = p\big(e^t\big).
\end{equation}
Since $\widehat{f}(A)=0$ when $|A|>d$, we get that
\begin{equation}
\begin{split}
\|e^{t\Delta}f\|_{L_p(\{-1,1\}^n;X)}& \stackrel{\eqref{ide2}}{=} \Big\| \int_{\overline{\Omega}} w^\Delta f \diff\rho_t(w)\Big\|_{L_p(\{-1,1\}^n;X)} \leq \int_{\overline{\Omega}} \|w^\Delta f\|_{L_p(\{-1,1\}^n;X)} \diff|\rho_t|(w)
\\ & \stackrel{\eqref{eq:assumppropKconv0}}{\leq} K \int_{\overline{\Omega}} \|f\|_{L_p(\{-1,1\}^n;X)} \diff|\rho_t|(w) = K\|\rho_t\|_{\ms{M}(\overline{\Omega})} \|f\|_{L_p(\{-1,1\}^n;X)}
\\ & \stackrel{\eqref{ineq2}}{\leq} K\big| \phi_{\Omega^c}\big(e^{t}\big)\big|^d \|f\|_{L_p(\{-1,1\}^n;X)},
\end{split}
\end{equation}
which is equivalent to \eqref{eq:reverseheatconformal}.
\end{proof}

\noindent {\it Proof of Theorem \ref{thm:reverseheatKconvex}.} Fix $n\in\N$, $d\in\{1,\ldots,n\}$, $p\in(1,\infty)$ and let $f:\{-1,1\}^n\to X$ be a function of degree at most $d$. By Theorem \ref{thm:reverseheatgeneral}, for every $t\geq0$,
\begin{equation}
\|e^{-t\Delta}f\|_{L_p(\{-1,1\}^n;X)} \geq \frac{1}{T_d(e^t)} \|f\|_{L_p(\{-1,1\}^n;X)}
\end{equation}
and, by \eqref{eq:chebyshevidentity}, for $t\geq1$,
\begin{equation}
T_d(e^t) \stackrel{\eqref{eq:chebyshevidentity}}{\leq} (2e^t)^d \leq e^{2td}.
\end{equation}
We will now treat the range $t\leq1$. By Corollary \ref{cor:lens}, there exists $r=r(p,X)\in[1,\infty)$ such that
\begin{equation}
\sup_{w\in\Omega(r)} \sup_{n\in\N} \|w^\Delta\|_{{L_p(\{-1,1\}^n;X)}\to {L_p(\{-1,1\}^n;X)}} \leq K,
\end{equation}
where $\Omega(r)$ is the domain \eqref{eq:Omega(r)def}. Therefore, by Proposition \ref{prop:reverseheatcrucial}, we have
\begin{equation}
\forall \ t\geq0, \ \ \ \|e^{-t\Delta}f\|_{L_p(\{-1,1\}^n;X)} \geq \frac{1}{K\big|\phi_{\Omega(r)^c}\big(e^{t}\big)\big|^d} \|f\|_{L_p(\{-1,1\}^n;X)}.
\end{equation}
To conclude the proof, we will show that there exists $C=C(p,X)$ such that
\begin{equation} \label{eq:boundconformalbypower2}
\forall \ 0\leq t\leq 1, \ \ \ \big|\phi_{\Omega(r)^c}\big(e^{t}\big)\big| \leq e^{C t^{\pi/(2\pi-\theta(r))}},
\end{equation}
where $\theta(r) = 2\arcsin(1/r)\in(0,\pi]$ which would then imply the conclusion of Theorem \ref{thm:reverseheatKconvex} with $\eta(p,X) = \frac{\pi}{2\pi-\theta(r(p,X))} \in \big(\frac{1}{2},1\big]$. By Lemma \ref{lem:computeconformal}, if $\theta=\theta(r)$, we have
\begin{equation} \label{eq:phiomegacr}
\phi_{\Omega(r)^c}\big(e^t\big) = \frac{(e^t+1)^{\frac{\pi}{2\pi-\theta}}+(e^t-1)^{\frac{\pi}{2\pi-\theta}}}{(e^t+1)^{\frac{\pi}{2\pi-\theta}}-(e^t-1)^{\frac{\pi}{2\pi-\theta}}}.
\end{equation}
Therefore,
\begin{equation}
\log \big|\phi_{\Omega(r)^c}\big(e^t\big)\big| = \log\left(1 + \frac{2(e^t-1)^{\frac{\pi}{2\pi-\theta}}}{(e^t+1)^{\frac{\pi}{2\pi-\theta}}-(e^t-1)^{\frac{\pi}{2\pi-\theta}}} \right) \leq \frac{2(e^t-1)^{\frac{\pi}{2\pi-\theta}}}{(e^t+1)^{\frac{\pi}{2\pi-\theta}}-(e^t-1)^{\frac{\pi}{2\pi-\theta}}}
\end{equation}
and for $t\leq1$,
\begin{equation}
\frac{2(e^t-1)^{\frac{\pi}{2\pi-\theta}}}{(e^t+1)^{\frac{\pi}{2\pi-\theta}}-(e^t-1)^{\frac{\pi}{2\pi-\theta}}} \leq \frac{2(e^t-1)^{\frac{\pi}{2\pi-\theta}}}{(e+1)^{\frac{\pi}{2\pi-\theta}}-(e-1)^{\frac{\pi}{2\pi-\theta}}} \leq \frac{2(e-1)^{\frac{\pi}{2\pi-\theta}}}{(e+1)^{\frac{\pi}{2\pi-\theta}}-(e-1)^{\frac{\pi}{2\pi-\theta}}} t^{\frac{\pi}{2\pi-\theta}},
\end{equation}
which implies \eqref{eq:boundconformalbypower2} with $C= \frac{2(e-1)^{\frac{\pi}{2\pi-\theta}}}{(e+1)^{\frac{\pi}{2\pi-\theta}}-(e-1)^{\frac{\pi}{2\pi-\theta}}}$.
\hfill$\Box$

\medskip

An argument identical to the one used in the proof of Corollary \ref{cor:realchaos} implies the following improved moment comparison of low degree functions with values in a $K$-convex space $(X,\|\cdot\|_X)$.

\begin{corollary} \label{cor:momentcompKconvex}
Let $(X,\|\cdot\|_X)$ be a $K$-convex Banach space. For every $p>q>1$, there exist $C=C(p,X)\in(0,\infty)$ and $\eta=\eta(p,X)\in\big(\frac{1}{2},1\big]$ such that for every $n,d\in\N$ with $d\in\{1,\ldots,n\}$ and every function $f:\{-1,1\}^n\to X$ of degree at most $d$, we have
\begin{equation}
\|f\|_{L_p(\{-1,1\}^n;X)} \leq Ce^{Cd\max\big\{ \log\big(\frac{p-1}{q-1}\big), \log \big(\frac{p-1}{q-1} \big)^\eta \big\}}\|f\|_{L_q(\{-1,1\}^n;X)}.
\end{equation}
\end{corollary}

We now proceed to prove Theorem \ref{thm:laplacianKconvex}, the improved Bernstein--Markov inequality for the Laplacian. For this, we will use the following \mbox{classical approximation theoretic result of Szeg\"o \cite{Sze25}.}

\begin{theorem} [Szeg\"o]
Let $\Omega \subseteq\C$ be a domain bounded by a Jordan curve $\gamma$ and fix $w\in\gamma$. Denote by $\theta(\Omega,w) \in [0,2\pi]$ the exterior angle of $\gamma$ at $w$. Then, there exists a constant $K(\Omega,w)\in(0,\infty)$ such that for every $d\in\N$ and every polynomial $p$ of degree at most $d$, we have
\begin{equation} \label{eq:szego}
|p'(w)| \leq K(\Omega,w) d^{\theta(\Omega,w)/\pi} \|p\|_{\ms{C}(\overline{\Omega})}.
\end{equation}
\end{theorem}

\noindent A simple proof of Szeg\"o's theorem specifically for lens domains of\mbox{ the form \eqref{eq:Omega(r)def} was given in \cite{EI18}.}

\medskip

\noindent {\it Proof of Theorem \ref{thm:laplacianKconvex}.} Fix $n\in\N$, $d\in\{1,\ldots,n\}$, $p\in(1,\infty)$ and let $f:\{-1,1\}^n\to X$ be a function of degree at most $d$. Also let $\theta=\theta(p,X)\in\big(0,\frac{\pi}{2}\big]$ and $M=M(p,X)\in[1,\infty)$ be given by Theorem \ref{thm:pisKconv}. Consider the domain $V = \{ e^{-z}: \ |\arg z|<\theta\} \subseteq\mb{D}$. Then, by Theorem \ref{thm:pisKconv}, we have
\begin{equation} \label{eq:usecomplexdomain}
\forall \ w\in V, \ \ \ \|w^\Delta\|_{{L_p(\{-1,1\}^n;X)}\to {L_p(\{-1,1\}^n;X)}} \leq M.
\end{equation}
As in the proof of Theorem \ref{thm:laplaciangeneral}, consider the linear functional $\psi:\big(\mathrm{span}\{1,w,\ldots,w^d\}, \|\cdot\|_{\ms{C}(\overline{V})}\big)\to \C$ given by
\begin{equation}
\psi\Big( \sum_{k=0}^d a_k w^k\Big) \eqdef \sum_{k=0}^d ka_k,
\end{equation}
or $\psi(p) = p'(1)$ when $p$ is a polynomial with $\mathrm{deg}(p)\leq d$. Notice that $V$ is a   domain bounded by a Jordan curve $\gamma$ with $1\in\gamma$ which forms an exterior angle $2\pi-2\theta$ at 1. Then, Szeg\"o's inequality \eqref{eq:szego} implies that there exists some $K=K(p,X)\in(0,\infty)$ such that
\begin{equation} \label{ineq3}
\forall \ p\in\mathrm{span}\{1,w,\ldots,w^d\}, \ \ \ |\psi(p)| \leq K d^{2-\frac{2\theta}{\pi}} \|p\|_{\ms{C}(\overline{V})}.
\end{equation}
Therefore, by the Hahn--Banach theorem and the Riesz representation theorem, there exists a complex measure $\nu$ on $\overline{V}$ such that $\|\nu\|_{\ms{M}(\overline{V})} \leq K d^{2-\frac{2\theta}{\pi}}$ such that for every polynomial $p$, we have
\begin{equation} \label{ide3}
\mathrm{deg}(p)\leq d \ \ \ \Longrightarrow \ \ \ \int_{\overline{V}} p(w)\diff\nu(w) = p'(1).
\end{equation}
Since $\widehat{f}(A)=0$ when $|A|>d$, we get that
\begin{equation}
\begin{split}
 \|\Delta f\|_{L_p(\{-1,1\}^n;X)}  \stackrel{\eqref{ide3}}{=} \Big\|& \int_{\overline{V}} w^\Delta f \diff\nu(w)\Big\|_{L_p(\{-1,1\}^n;X)} \leq \int_{\overline{V}} \|w^\Delta f\|_{L_p(\{-1,1\}^n;X)} \diff|\nu|(w) 
\\ & \stackrel{\eqref{eq:conclusionKconv}}{\leq} M \int_{\overline{V}} \|f\|_{L_p(\{-1,1\}^n;X)} \diff|\nu|(w) \stackrel{\eqref{ineq3}}{\leq} KM d^{2-\frac{2\theta}{\pi}} \|f\|_{L_p(\{-1,1\}^n;X)},
\end{split}
\end{equation}
which concludes the proof of the theorem with $\alpha(p,X) = 2-\frac{2\theta(p,X)}{\pi} \in [1,2)$.
\hfill$\Box$

\begin{remark} \label{rem:laplacianimpliesreverseheat}
It is straightfoward to see that if the Bernstein--Markov inequality \eqref{eq:laplaciangeneral} holds true with linear dependence on the degree for a given Banach space, then the asymptotically optimal lower bound for the action of the heat semigroup conjectured after Theorem \ref{thm:reverseheatKconvex} follows. Indeed, assume that for a Banach space $(X,\|\cdot\|_X)$ and $p\in[1,\infty)$ there exists $C\in(0,\infty)$ such that for every $n,d\in\N$ with $d\in\{1,\ldots,n\}$ and every function $f:\{-1,1\}^n\to X$ of degree at most $d$, we have
\begin{equation} \label{eq:assumptioninremark}
 \|\Delta f\|_{L_p(\{-1,1\}^n;X)}  \leq Cd  \|f\|_{L_p(\{-1,1\}^n;X)}.
\end{equation}
Then, for every $t\geq0$, we get
\begin{equation} \label{eq:franc}
\begin{split}
\|e^{t\Delta} f\|_{L_p(\{-1,1\}^n;X)} \leq \sum_{k=0}^\infty \frac{\|(t\Delta)^k f\|_{L_p(\{-1,1\}^n;X)}}{k!} \stackrel{\eqref{eq:assumptioninremark}}{\leq} \sum_{k=0}^\infty \frac{(Ctd)^k}{k!} & \|f\|_{L_p(\{-1,1\}^n;X)} \\ & = e^{Ctd}\|f\|_{L_p(\{-1,1\}^n;X)},
\end{split}
\end{equation}
which, applied to $e^{-t\Delta}f$, is equivalent to the conjectured optimal version of Theorem \ref{thm:reverseheatKconvex}.
\end{remark}

The Bernstein--Markov inequality \eqref{eq:gradientKconvex} for the vector valued gradient is an immediate consequence of Theorem \ref{thm:laplacianKconvex} combined with the following dual to Pisier's inequality \eqref{eq:Pisierlowfreq} for low degree functions.

\begin{proposition}
Let $(X,\|\cdot\|_X)$ be a $K$-convex Banach space. For every $p\in(1,\infty)$ there exists $B(p,X)\in(0,\infty)$ such that for every $n,d\in\N$ with $d\in\{1,\ldots,n\}$, every function $g:\{-1,1\}^n\to X$ of degree at most $d$ satisfies
\begin{equation} \label{eq:dualpisier}
\Big(\frac{1}{2^n}\sum_{\delta\in\{-1,1\}^n} \Big\| \sum_{i=1}^n\delta_i\partial_i g\Big\|^p_{L_p(\{-1,1\}^n;X)} \Big)^{1/p} \leq B(p,X)(\log d+1) \|\Delta g\|_{L_p(\{-1,1\}^n;X)}.
\end{equation}
\end{proposition}

\begin{proof}
Following the notation \eqref{eq:takeradproj} of the proof of Theorem \ref{thm:Pisierlowfreq}, for $t\geq0$, consider the function $g_t:\{-1,1\}^n\times\{-1,1\}^n\to X$ given by
\begin{equation} \label{xanarad}
\begin{split}
g_t(\e,\delta) & \eqdef \sum_{A\subseteq\{1,\ldots,n\}} \widehat{g}(A)  \prod_{i\in A} \big( e^{-t} \e_i + (1-e^{-t}) \delta_i\big)
\\ & = e^{-t\Delta} g(\e) + (e^t-1) \sum_{i=1}^n \e_i\delta_i\partial_i e^{-t\Delta}g(\e) + \Phi_{t}(\e,\delta),
\end{split}
\end{equation}
where $\sum_{\delta\in\{-1,1\}^n} \delta_i \Phi_{t}(\e,\delta)=0$ for every $\e\in\{-1,1\}^n$ and $i\in\{1,\ldots,n\}$. Therefore, for every $t>0$, we have
\begin{equation} \label{eq:dualpisier1}
\begin{split}
\Big(\frac{1}{2^n}\sum_{\delta\in\{-1,1\}^n} \Big\| \sum_{i=1}^n\delta_i\partial_i e^{-t\Delta} g&\Big\|^p_{L_p(\{-1,1\}^n;X)} \Big)^{1/p} 
 \stackrel{\eqref{xanarad}}{=} \frac{1}{e^t-1} \big\| \msf{Rad}_\delta g_t \big\|_{L_p(\{-1,1\}^n\times\{-1,1\}^n;X)}
\\ & \leq \frac{K}{e^t-1} \|g_t\|_{L_p(\{-1,1\}^n\times\{-1,1\}^n;X)} \stackrel{\eqref{eq:twovarcontract2}}{\leq} \frac{K}{e^t-1} \|g\|_{L_p(\{-1,1\}^n;X)},
\end{split}
\end{equation}
where $\msf{Rad}_\delta g_t$ is the Rademacher projection of $g_t$ with respect to the variable $\delta\in\{-1,1\}^n$, $K=K(p,X) = \sup_{n\in\N} \|\msf{Rad}\|_{L_p(\{-1,1\}^n;X)\to L_p(\{-1,1\}^n;X)} < \infty$ and the proof of the last inequality is identical to the proof of \eqref{eq:twovarcontract2} via \eqref{eq:twovarcontract1}. Integrating the above inequality, we deduce that for $s>0$,
\begin{equation} \label{eq:dualpisier1.5}
\begin{split}
\Big(\frac{1}{2^n}\sum_{\delta\in\{-1,1\}^n} \Big\| &\sum_{i=1}^n\delta_i\partial_i \Delta^{-1} e^{-s\Delta} g\Big\|_{L_p(\{-1,1\}^n;X)}^p \Big)^{1/p} 
\\ & = \Big(\frac{1}{2^n}\sum_{\delta\in\{-1,1\}^n} \Big\| \int_s^\infty \sum_{i=1}^n\delta_i\partial_i e^{-t\Delta} g \diff t\Big\|^p_{L_p(\{-1,1\}^n;X)} \Big)^{1/p} 
\\ & \stackrel{\eqref{eq:dualpisier1}}{\leq} K \Big(\int_s^\infty \frac{1}{e^t-1}\diff t\Big) \|g\|_{L_p(\{-1,1\}^n;X)} = K \log\Big(\frac{e^s}{e^s-1}\Big) \|g\|_{L_p(\{-1,1\}^n;X)}.
\end{split}
\end{equation}
By Theorem \ref{thm:reverseheatgeneral} and the elementary inequality $T_d(e^s)\leq e^{d^2s}$, where $s\geq0$, we conclude that for every $\delta\in\{-1,1\}^n$,
\begin{equation} \label{eq:dualpisier2}
\Big\| \sum_{i=1}^n\delta_i\partial_i \Delta^{-1} e^{-s\Delta} g\Big\|_{L_p(\{-1,1\}^n;X)} \geq e^{-d^2s} \Big\| \sum_{i=1}^n\delta_i\partial_i \Delta^{-1} g\Big\|_{L_p(\{-1,1\}^n;X)}.
\end{equation}
Therefore, combining \eqref{eq:dualpisier1.5} and \eqref{eq:dualpisier2},
\begin{equation}
\begin{split}
\Big(\frac{1}{2^n}\sum_{\delta\in\{-1,1\}^n} \Big\| \sum_{i=1}^n\delta_i\partial_i\Delta^{-1} g\Big\|_{L_p(\{-1,1\}^n;X)}^p \Big)^{1/p} \leq K \min_{s\geq0}& e^{d^2s}\log\Big(\frac{e^s}{e^s-1}\Big)  \|g\|_{L_p(\{-1,1\}^n;X)}
\\ & \leq 3K(\log d+1) \|g\|_{L_p(\{-1,1\}^n;X)},
\end{split}
\end{equation}
which is equivalent to the desired inequality \eqref{eq:dualpisier} with $B(p,X)=3K(p,X)$.
\end{proof}

\begin{remark}
It has been shown in \cite{NS02} that the validity of \eqref{eq:dualpisier} with the factor $C(\log d+1)$ replaced by a constant $C(p,X)$ depending only on $p$ and the Banach space $(X,\|\cdot\|_X)$, where $p\in(1,\infty)$, implies that $X$ is $K$-convex. Nevertheless, \eqref{eq:dualpisier} is the best known bound to date for general $K$-convex spaces, even when $d=n$. Under additional assumptions (e.g. when $X$ is a UMD$^+$ space or when $X$ is a $K$-convex Banach lattice), inequality \eqref{eq:dualpisier} is known to hold true with a constant $C(p,X)$ independent of the dimension $n$ for functions of arbitrary degree $d$, see \cite{HN13}.
\end{remark}

\noindent {\it Proof of Theorem \ref{thm:gradientKconvex}.} Fix $n\in\N$, $d\in\{1,\ldots,n\}$ and let $f:\{-1,1\}^n\to X$ be a function of degree at most $d$. Then, we have
\begin{equation}
\begin{split}
\Big(\frac{1}{2^n}\sum_{\delta\in\{-1,1\}^n} \Big\| \sum_{i=1}^n\delta_i\partial_i f\Big\|&^p_{L_p(\{-1,1\}^n;X)}  \Big)^{1/p} \stackrel{\eqref{eq:dualpisier}}{\leq} B(p,X)(\log d+1) \|\Delta f\|_{L_p(\{-1,1\}^n;X)}
\\ & \stackrel{\eqref{eq:laplacianKconvex}}{\leq} B(p,X) C(p,X) d^{\alpha(p,X)} (\log d+1)  \|f\|_{L_p(\{-1,1\}^n;X)},
\end{split}
\end{equation}
which completes the proof.
\hfill$\Box$

\medskip

Finally, we will prove the reverse Bernstein--Markov inequality of Theorem \ref{thm:reversebernsteinKconvex}. Recall that for $\gamma\in(0,\infty)$ and a function $f:\{-1,1\}^n\to X$, we denote by
\begin{equation}
\Delta^\gamma f \eqdef \sum_{A\subseteq\{1,\ldots,n\}} |A|^\gamma \widehat{f}(A) w_A
\end{equation}
the action of a fractional power of the hypercube Laplacian $\Delta$ on $f$. We will prove the following statement for $\Delta^{1/2}$ of which Theorem \ref{thm:reversebernsteinKconvex} is an immediate consequence.

\begin{theorem} \label{thm:delta1/2}
Let $(X,\|\cdot\|_X)$ be a $K$-convex Banach space. For every $p\in(1,\infty)$ there exists $c(p,X)\in(0,\infty)$ such that for every $n,d,m\in\N$ with $d+m\leq n$ and every function $f:\{-1,1\}^n\to X$ of degree at most $d+m$ which is also in the $d$-th tail space, we have
\begin{equation} \label{eq:delta1/2}
\|\Delta^{1/2} f\|_{L_p(\{-1,1\}^n;X)} \geq c(p,X) \sqrt{\frac{d}{m}} \|f\|_{L_p(\{-1,1\}^n;X)}.
\end{equation}
\end{theorem}

We start by proving the following lemma.

\begin{lemma} \label{lem:cauchyformula}
Let $(X,\|\cdot\|_X)$ be a $K$-convex Banach space. For every $p\in(1,\infty)$ there exists $K=K(p,X)\in(0,\infty)$ such that for every $n\in\N$ and every function $f:\{-1,1\}^n\to X$, we have
\begin{equation} \label{eq:cauchyformula}
\forall \ t\geq0, \ \ \  \|\Delta^{1/2} e^{-t\Delta}f\|_{L_p(\{-1,1\}^n;X)} \leq \frac{K}{\sqrt{e^{2t}-1}} \|f\|_{L_p(\{-1,1\}^n;X)}.
\end{equation}
\end{lemma}

\begin{proof}
By \eqref{eq:boundRadk}, there exist $M=M(p,X)\in(0,\infty)$ and $a=a(p,X)\in(0,\infty)$ such that for every $n\in\N$, $k\in\{1,\ldots,n\}$ and every function $f:\{-1,1\}^n\to X$, we have
\begin{equation} \label{eq:usenow}
\big\|\msf{Rad}_kf\big\|_{L_p(\{-1,1\}^n;X)} \leq M e^{ka} \|f\|_{L_p(\{-1,1\}^n;X)}.
\end{equation}
Therefore, for $t\geq 2a$,
\begin{equation}
\begin{split}
\|\Delta^{1/2}& e^{-t\Delta} f\|_{L_p(\{-1,1\}^n;X)}  \leq \sum_{k=1}^n \sqrt{k} e^{-tk}\big\|\msf{Rad}_kf\big\|_{L_p(\{-1,1\}^n;X)}
\\ & \stackrel{\eqref{eq:usenow}}{\leq} M \sum_{k=1}^n \sqrt{k} e^{-(t-a)k} \|f\|_{L_p(\{-1,1\}^n;X)} \leq \Big(Me^a \sum_{k=1}^\infty \sqrt{k} e^{-ak}\Big) e^{-t} \|f\|_{L_p(\{-1,1\}^n;X)},
\end{split}
\end{equation}
which implies that for $t\geq 2a$,
\begin{equation}
 \|\Delta^{1/2} e^{-t\Delta}f\|_{L_p(\{-1,1\}^n;X)} \leq \frac{K}{\sqrt{e^{2t}-1}} \|f\|_{L_p(\{-1,1\}^n;X)},
\end{equation}
where $K=Me^a \sum_{k=1}^\infty \sqrt{k} e^{-ak}$. To prove \eqref{eq:cauchyformula} for $t\in(0,2a)$, recall that by Pisier's $K$-convexity theorem, there exists $\theta=\theta(p,X)\in\big(0,\frac{\pi}{2}\big]$ such that
\begin{equation} \label{eq:pisagain}
|\arg z|\leq \theta \ \ \ \Longrightarrow \ \ \ \|e^{-z\Delta}\|_{{L_p(\{-1,1\}^n;X)}\to {L_p(\{-1,1\}^n;X)}} \leq M.
\end{equation}
Let $r = t\sin\theta$ and notice that the closed disc $\overline{D}(t,r)$ of radius $r$ centered at $t$ is contained in $\{z\in\C: \ |\arg z|\leq\theta\}$. By the Cauchy integral formula for the derivative,
\begin{equation} \label{eq:cauchywritten}
-\Delta e^{-t\Delta} = \frac{1}{2\pi i} \int_{\partial D(t,r)} \frac{e^{-\zeta \Delta}}{(\zeta-t)^2}\diff \zeta,
\end{equation}
hence
\begin{equation} \label{eq:boundDetD}
\begin{split}
\|\Delta e^{-t\Delta}f\|_{L_p(\{-1,1\}^n;X)} \stackrel{\eqref{eq:cauchywritten}}{\leq} \frac{1}{r} \sup_{\zeta \in \partial D(t,r)} \|e^{-\zeta\Delta}f\|_{L_p(\{-1,1\}^n;X)} \stackrel{\eqref{eq:pisagain}}{\leq} \frac{M}{t\sin\theta} \|f\|_{L_p(\{-1,1\}^n;X)}.
\end{split}
\end{equation}
Furthermore, by an inequality of Naor and Schechtman \cite[Lemma~5.6]{BELP08}, for every $\beta\in(0,1)$ and function $g:\{-1,1\}^n\to X$, we have
\begin{equation} \label{eq:naos}
\|\Delta^{\beta}g\|_{L_p(\{-1,1\}^n;X)} \leq 4 \|\Delta g\|_{L_p(\{-1,1\}^n;X)}^{\beta}\|g\|_{L_p(\{-1,1\}^n;X)}^{1-\beta}.
\end{equation}
Therefore, combining \eqref{eq:naos} for $g=e^{-t\Delta}f$ and $\beta=\frac{1}{2}$ with \eqref{eq:boundDetD} and the contractivity of the heat semigroup, we deduce that
\begin{equation}
\|\Delta^{1/2} e^{-t\Delta} f\|_{L_p(\{-1,1\}^n;X)} \leq \frac{4\sqrt{M}}{\sqrt{t}\sqrt{\sin\theta}}\|f\|_{L_p(\{-1,1\}^n;X)},
\end{equation}
which completes the proof of \eqref{eq:cauchyformula}, since for every $a\in(0,\infty)$, there exists $c_a\in(0,\infty)$ such that $\sqrt{e^{2t}-1} \leq c_a \sqrt{t}$ for $t\in(0,2a)$.
\end{proof}

In the proof of Theorem \ref{thm:delta1/2}, we will use a reverse Bernstein inequality for incomplete polynomials, proven recently by Erd\'elyi \cite{Erd18}.

\begin{theorem} [Erd\'elyi] \label{thm:erd}
Fix $d,m\in\N$ and let $P(x)$ be a polynomial of the form 
\begin{equation}
P(x) = a_d x^d + a_{d+1}x^{d+1} +\cdots + a_{d+m} x^{d+m}.
\end{equation}
Then,
\begin{equation} \label{eq:erdelyi}
|P(1)| \leq 6 \sqrt{\frac{m}{d}} \|\sqrt{1-x^2}P'(x)\|_{\ms{C}([0,1])}.
\end{equation}
Furthermore, the estimate is sharp up to the value of the universal constant.
\end{theorem}

\noindent {\it Proof of Theorem \ref{thm:delta1/2}.}
Fix $n,d,m\in\N$ with $d+m\leq n$ and a function $f:\{-1,1\}^n\to X$ of degree at most $d+m$ which is also in the $d$-th tail space. Writing $x=e^{-t}\in[0,1]$,\mbox{ \eqref{eq:cauchyformula} can be rewritten as}
\begin{equation} \label{eq:rewritecauchy}
\forall \ x\in[0,1], \ \ \ \|\sqrt{1-x^2}\Delta x^{\Delta-1}f\|_{L_p(\{-1,1\}^n;X)} \leq K \|\Delta^{1/2}f\|_{L_p(\{-1,1\}^n;X)}.
\end{equation}
Consider the linear functional $\xi:\big(\mathrm{span}\big\{\sqrt{1-x^2}x^{d-1},\ldots,\sqrt{1-x^2}x^{d+m-1}\big\},\|\cdot\|_{\ms{C}([0,1])}\big)$ given by
\begin{equation}
\xi\Big(  \sqrt{1-x^2} \sum_{k=d}^{d+m} k a_k x^{k-1}\Big) \eqdef \sum_{k=d}^{d+m} a_k.
\end{equation}
Then, Erd\'elyi's inequality \eqref{eq:erdelyi} can be rewritten as
\begin{equation} \label{ineq4}
\forall \ p\in\mathrm{span}\big\{\sqrt{1-x^2}x^{d-1},\ldots,\sqrt{1-x^2}x^{d+m-1}\big\},  \ \ \ |\xi(p)| \leq 6\sqrt{\frac{m}{d}} \|p\|_{\ms{C}([0,1])}.
\end{equation}
Therefore, by the Hahn--Banach theorem and the Riesz representation theorem, there exists a complex measure $\sigma$ on $[0,1]$ such that $\|\sigma\|_{\ms{M}([0,1])}\leq 6\sqrt{\frac{m}{d}}$ and
\begin{equation} \label{ide4}
\forall \ p\in\mathrm{span}\big\{\sqrt{1-x^2}x^{d-1},\ldots,\sqrt{1-x^2}x^{d+m-1}\big\},  \ \ \ \int_0^1 p(x)\diff\sigma(x) = \xi(p).
\end{equation}
Since $\widehat{f}(A)=0$ when $|A|\notin\{d,\ldots,d+m\}$, we deduce that
\begin{equation}
\begin{split}
\|f\|&_{L_p(\{-1,1\}^n;X)} \stackrel{\eqref{ide4}}{=} \Big\| \int_0^1 \sqrt{1-x^2} \Delta x^{\Delta-1} f\diff\sigma(x)\Big\|_{L_p(\{-1,1\}^n;X)} 
\\ & \leq \int_0^1 \big\|\sqrt{1-x^2} \Delta x^{\Delta-1} f\big\|_{L_p(\{-1,1\}^n;X)}\diff|\sigma|(x) \stackrel{\eqref{eq:rewritecauchy}}{\leq} \int_0^1 K\|\Delta^{1/2}f\|_{L_p(\{-1,1\}^n;X)} \diff|\sigma|(x)
\\ & = K\|\sigma\|_{\ms{M}([0,1])} \|\Delta^{1/2}f\|_{L_p(\{-1,1\}^n;X)} \stackrel{\eqref{ineq4}}{\leq} 6K\sqrt{\frac{m}{d}} \|\Delta^{1/2}f\|_{L_p(\{-1,1\}^n;X)},
\end{split}
\end{equation}
which is equivalent to \eqref{eq:delta1/2}.
\hfill$\Box$

\medskip

\noindent {\it Proof of Theorem \ref{thm:reversebernsteinKconvex}.}
Fix $n,d,m\in\N$ with $d+m\leq n$ and a function $f:\{-1,1\}^n\to X$ of degree at most $d+m$ which is also in the $d$-th tail space. Then, Theorem \ref{thm:delta1/2} implies that
\begin{equation}
\begin{split}
\|\Delta f\|_{L_p(\{-1,1\}^n;X)} & = \|\Delta^{1/2}\Delta^{1/2}f\|_{L_p(\{-1,1\}^n;X)} \\ & \stackrel{\eqref{eq:delta1/2}}{\geq} c(p,X) \sqrt{\frac{d}{m}} \|\Delta^{1/2}f\|_{L_p(\{-1,1\}^n;X)} \stackrel{\eqref{eq:delta1/2}}{\geq} c(p,X)^2 \frac{d}{m} \|f\|_{L_p(\{-1,1\}^n;X)},
\end{split}
\end{equation}
which is the desired inequality.
\hfill$\Box$

\section{Estimates for scalar valued functions} \label{sec:4}

We noticed in Corollary \ref{cor:lens} that Pisier's $K$-convexity theorem easily implies that for every $K$-convex Banach space $(X,\|\cdot\|_X)$ and $p\in(1,\infty)$ there exist $r=r(p,X)\in[1,\infty)$ and $K=K(p,X)\in[1,\infty)$ such that
\begin{equation}
\sup_{w\in\Omega(r)}\sup_{n\in\mathbb{N}}\|w^{\Delta}\|_{{L_p(\{-1,1\}^n;X)}\to {L_p(\{-1,1\}^n;X)}} \leq K,
\end{equation}
where the domain $\Omega(r)$ is given by \eqref{eq:Omega(r)def}. The scalar valued version of this result was first studied in classical work of Weissler \cite{Wei79} who found the exact domain $\Omega_p\subseteq\mb{D}$ for which  the operator $w^\Delta:L_p(\{-1,1\}^n;\C)\to L_p(\{-1,1\}^n;\C)$ is uniformly bounded (equivalently, a contraction) for $p\in(1,\infty)\setminus \big(\frac{3}{2},2\big)\cup(2,3)$ and $w\in\Omega_p$. Finally, the domain $\Omega_p$ for $p$ in the remaining range $\big(\frac{3}{2},2\big)\cup(2,3)$ was recently identified by the second named author and Nazarov in \cite{IN19}.

\begin{theorem} [Weissler, Ivanisvili--Nazarov] \label{thm:weissler}
For $p\in(1,\infty)$, let $r_p=\frac{p}{2\sqrt{p-1}}$. Then, for every $p\in(1,\infty)$ and $w\in\mb{C}$, we have
\begin{equation} \label{eq:weissler}
w\in\overline{\Omega(r_p)} \ \ \Longleftrightarrow \ \  \sup_{n\in\mb{N}} \|w^{\Delta}\|_{{L_p(\{-1,1\}^n;\C)}\to {L_p(\{-1,1\}^n;\C)}}=1.
\end{equation}
Furthermore, both conditions are equivalent to
\begin{equation}
\sup_{n\in\mb{N}} \|w^{\Delta}\|_{{L_p(\{-1,1\}^n;\C)}\to {L_p(\{-1,1\}^n;\C)}} <\infty.
\end{equation}
\end{theorem}

The decay properties of Theorem \ref{thm:smoothingR} are straightforward consequences of Theorem \ref{thm:weissler} and the results of Section \ref{sec:3}.

\medskip

\noindent {\it Proof of Theorem \ref{thm:smoothingR}.}
Inequality \eqref{eq:lowersmoothingR} for functions of low degree is a straightforward consequence of Theorem \ref{thm:weissler} combined with Proposition \ref{prop:reverseheatcrucial} and \eqref{eq:phiomegacr}. Similarly, inequality \eqref{eq:uppersmoothingR} for functions in the tail space follows from Theorem \ref{thm:weissler} combined with Proposition \ref{prop:crucialheat} and \eqref{eq:phiomegar}.
\hfill$\Box$

\medskip

\noindent {\it Proof of Corollary \ref{cor:momentcompR}.}
The deduction of Corollary \ref{cor:momentcompR} from Theorem \ref{thm:smoothingR} is identical to the deduction of Corollary \ref{cor:realchaos} from Theorem \ref{thm:reverseheatgeneral}; it follows by concatenating \eqref{eq:lowersmoothingR} with Bonami's hypercontractive inequality \eqref{eq:bonami} and choosing $t=\frac{1}{2}\log\big(\frac{p-1}{q-1}\big)$.
\hfill$\Box$

\medskip

We now proceed with the proof of Theorem \ref{thm:momentcomp1}, the improved $L_1-L_2$ moment comparison for low degree functions. We use the following important result of Beckner \cite{Bec75} and Weissler \cite{Wei79}.

\begin{theorem} [Beckner, Weissler] \label{thm:beckner}
Let $p\in(2,\infty)$ and $p^\ast \in(1,2)$ its conjugate exponent, that is $p^\ast = \frac{p}{p-1}$. Then, a complex number $w\in\mb{D}$ satisfies
\begin{equation} \label{eq:becknerassumption}
\max\Big\{ \Big|w-\frac{p-2}{2(p-1)}\Big|, \Big|w+\frac{p-2}{2(p-1)}\Big| \Big\} \leq \frac{p}{2(p-1)} \Big\}
\end{equation}
if and only if
\begin{equation} \label{eq:becknerconclusion}
\sup_{n\in\mb{N}} \|w^{\Delta}\|_{{L_p(\{-1,1\}^n;\C)}\to {L_{p^\ast}(\{-1,1\}^n;\C)}} =1.
\end{equation}
Furthermore, both conditions are equivalent to
\begin{equation}
\sup_{n\in\mb{N}} \|w^{\Delta}\|_{{L_p(\{-1,1\}^n;\C)}\to {L_{p^\ast}(\{-1,1\}^n;\C)}} <\infty.
\end{equation}
\end{theorem}

In \cite{Bec75}, Beckner proved that $w=\pm i\sqrt{p^\ast-1}$ satisfies \eqref{eq:becknerconclusion} and then Weissler \cite{Wei79} modified his argument to obtain \eqref{eq:becknerconclusion} for $w$ in the full domain\eqref{eq:becknerassumption}.

\medskip

\noindent {\it Proof of Theorem \ref{thm:momentcomp1}.}
Fix $n\in\N$ and $d\in\{1,\ldots,n\}$. For $p\in(2,\infty)$, denote by $V_p\subseteq\mb{D}$ the set of all $w\in\mb{D}$ that satisfy \eqref{eq:becknerassumption}. Then, a straightforward computation shows that
\begin{equation}
i\sqrt{p-1} V_p = \overline{\Omega\Big(\frac{p}{2\sqrt{p-1}}\Big)},
\end{equation}
where $\Omega(r)$ is defined by \eqref{eq:Omega(r)def}. Therefore, by Lemma \ref{lem:computeconformal}, the function $\varphi_p: V_p^c\to\overline{\mb{D}}^c$
\begin{equation}
\forall \ w\in \mb{C}\setminus V_p, \ \ \ \varphi_p(w)\eqdef \frac{(i\sqrt{p-1}w-1)^{\frac{\pi}{2\pi-\theta_p}}+(i\sqrt{p-1}w+1)^{\frac{\pi}{2\pi-\theta_p}}}{(i\sqrt{p-1}w-1)^{\frac{\pi}{2\pi-\theta_p}}-(i\sqrt{p-1}w+1)^{\frac{\pi}{2\pi-\theta_p}}},
\end{equation}
where $\theta_p = 2\arcsin\big(\frac{2\sqrt{p-1}}{p}\big)$, is a conformal equivalence between the complement of $V_p$ and the complement of the closed unit disc $\overline{\mb{D}}$. A duality argument identical to that of Proposition \ref{prop:reverseheatcrucial} combined with \eqref{eq:becknerconclusion} now implies that for every function $f:\{-1,1\}^n\to \C$ of degree at most $d$,
\begin{equation}
\forall \ t\geq0, \ \ \ \|e^{-t\Delta}f\|_{L_{p^\ast}(\{-1,1\}^n;\C)} \geq \frac{1}{\big|\varphi_p\big(e^t\big)\big|^d} \|f\|_{L_p(\{-1,1\}^n;\C)}.
\end{equation}
In particular, for $t=0$, we have
\begin{equation} \label{eq:pp*}
\|f\|_{L_p(\{-1,1\}^n;\C)} \leq |\varphi_p(1)|^d \|f\|_{L_{p^\ast}(\{-1,1\}^n;\C)}.
\end{equation}
Now, by H\"older's inequality, we get
\begin{equation}
\|f\|_{L_2(\{-1,1\}^n;\C)}\leq \|f\|_{L_1(\{-1,1\}^n;\C)}^{\frac{p-2}{2(p-1)}} \|f\|_{L_p(\{-1,1\}^n;\C)}^{\frac{p}{2(p-1)}}
\end{equation}
and
\begin{equation}
\|f\|_{L_{p^\ast}(\{-1,1\}^n;\C)} \leq \|f\|_{L_1(\{-1,1\}^n;\C)}^{\frac{p-2}{p-1}} \|f\|_{L_p(\{-1,1\}^n;\C)}^{\frac{1}{p-1}}.
\end{equation}
Combining the two, we deduce that
\begin{equation}
\begin{split}
\|f\|_{L_2(\{-1,1\}^n;\C)} \leq \|f\|_{L_1(\{-1,1\}^n;\C)}  \left(\frac{\|f\|_{L_p(\{-1,1\}^n;\C)}}{\|f\|_{L_{p^\ast}(\{-1,1\}^n;\C)}}\right)&^{\frac{p}{2(p-2)}} 
\\ & \stackrel{\eqref{eq:pp*}}{\leq} |\varphi_p(1)|^{\frac{pd}{2(p-2)}} \|f\|_{L_1(\{-1,1\}^n;\C)}.
\end{split}
\end{equation}
Consequently,
\begin{equation}
\|f\|_{L_2(\{-1,1\}^n;\C)} \leq C^d \|f\|_{L_1(\{-1,1\}^n;\C)},
\end{equation}
where
\begin{equation}
C \eqdef \inf_{p>2} \left| \frac{(i\sqrt{p-1}-1)^{\frac{\pi}{2\pi-\theta_p}}+(i\sqrt{p-1}+1)^{\frac{\pi}{2\pi-\theta_p}}}{(i\sqrt{p-1}-1)^{\frac{\pi}{2\pi-\theta_p}}-(i\sqrt{p-1}+1)^{\frac{\pi}{2\pi-\theta_p}}} \right|^{\frac{p}{2(p-2)}} < 2.69076.
\end{equation}
The last inequality can be checked numerically.
\hfill$\Box$

\medskip

We note in passing that the bound $\|f\|_{L_2(\{-1,1\}^n;\mb{C})} \leq e^{d/2} \|f\|_{L_1(\{-1,1\}^n;\mb{C})}$, which improves upon Theorem \ref{thm:momentcomp1}, was obtained in the recent work \cite{IT18} for $d$-homogeneous functions $f:\{-1,1\}^n\to\C$.

\begin{remark} \label{rem:complexhyperc}
Following Beckner's pioneering work \cite{Bec75}, significant efforts were devoted in identifying the complex domains consisting of those $w\in\mb{D}$ for which
\begin{equation}
\sup_{n\in\mb{N}} \|w^{\Delta}\|_{{L_p(\{-1,1\}^n;\C)}\to {L_q(\{-1,1\}^n;\C)}} <\infty
\end{equation}
for general $p\geq q>1$. In \cite{Wei79}, Weissler managed to precisely characterize these complex numbers $w$ for all $p\geq q>1$ apart from the cases $\frac{3}{2}<q \leq p<2$ and $2< q\leq p<3$ and posed a conjecture for $p,q$ in the remaining ranges. The case $p=q$ of his conjecture was recently settled by the second named author and Nazarov \cite{IN19}. In contrast to this long standing problem, Epperson \cite{Epp89} (see also \cite{Jan97}) has characterized those $w\in\mb{D}$ for which
\begin{equation}
\sup_{n\in\mb{N}} \|w^{L}\|_{{L_p((\R^n,\gamma_n);\C)}\to {L_q((\R^n,\gamma_n);\C)}} <\infty,
\end{equation}
where $\gamma_n$ is the standard Gaussian measure on $\R^n$ and $L=\Delta-\langle x, \nabla\rangle$ is the generator of the Ornstein--Uhlenbeck semigroup, for every $p>q>1$. For $p$ and $q$ not belonging in the missing ranges mentioned earlier, the domains of complex hypercontractivity for the Hamming cube and the Gauss space coincide and it is natural to believe that this is also the case when $\frac{3}{2}<q< p<2$ and $2<q< p<3$.
\end{remark}

It is evident from the proof above that the constant 2.69076 appearing in Theorem \ref{thm:momentcomp1} is not optimal. In fact, one can run a similar argument starting with the inequality
\begin{equation}
\|f\|_{L_2(\{-1,1\}^n;\C)} \leq \|f\|_{L_1(\{-1,1\}^n;\C)}  \left(\frac{\|f\|_{L_p(\{-1,1\}^n;\C)}}{\|f\|_{L_{q}(\{-1,1\}^n;\C)}}\right)^{\frac{pq}{2(p-q)}},
\end{equation}
which is valid for any $p>2>q$ and any function $f:\{-1,1\}^n\to\C$. Then, to obtain an $L_q-L_p$ moment comparison as in the proof of Theorem \ref{thm:momentcomp1}, one should use the general $L_q-L_p$ complex hypercontractivity of \cite{Wei79} and explicitly compute the conformal map of the domain provided by Weissler's theorem. In fact, such a computation could also provide an improvement of Corollary \ref{cor:momentcompR}. We did not attempt to optimize any of these computations as the domains of $L_q-L_p$ hypercontractivity for $q\notin\{p,p^\ast\}$ tend to be quite complicated. We also note that for $p>2$, the least constant $C_p$ for which every function $f:\{-1,1\}^n\to\C$ of degree at most $d$ satisfies
\begin{equation} \label{sqrtp-1}
\|f\|_{L_p(\{-1,1\}^n;\C)} \leq C_p^{d} \|f\|_{L_2(\{-1,1\}^n;\C)}
\end{equation}
is known to be $C_p=\sqrt{p-1}$ (see \cite{IT18}), yet the sharp constant in \eqref{eq:momentcomp1} is still unknown. Inequality \eqref{sqrtp-1} with $C_p=\sqrt{p-1}$ is usually proven via an orthogonality argument (see \cite[Theorem~9.21]{O'D14}), but a duality based proof can be given using the result of Weissler \cite{Wei79} who showed that for $p\geq2$
\begin{equation} \label{www}
|w|\leq \frac{1}{\sqrt{p-1}} \ \ \ \Longleftrightarrow \ \ \ 
\sup_{n\in\mb{N}} \|w^{\Delta}\|_{{L_p(\{-1,1\}^n;\C)}\to {L_2(\{-1,1\}^n;\C)}}=1.
\end{equation}
A straightfoward adaptation of the proof of Proposition \ref{prop:reverseheatcrucial} then implies that for every function $f:\{-1,1\}^n\to\C$ of degree at most $d$, we have
\begin{equation}
\forall \ t\geq0, \ \ \ \|e^{-t\Delta}f\|_{L_2(\{-1,1\}^n;\C)} \geq \Big(\frac{e^{-t}}{\sqrt{p-1}}\Big)^d \|f\|_{L_p(\{-1,1\}^n;\C)},
\end{equation}
which for $t=0$ coincides with \eqref{sqrtp-1} with $C_p=\sqrt{p-1}$.
\medskip

\noindent {\it Proof of Theorem \ref{thm:laplacianR}.}
The proof is a mechanical adaptation of the proof of Theorem \ref{thm:laplacianKconvex}, where \eqref{eq:usecomplexdomain} is replaced by the characterization \eqref{eq:weissler} of the complex domain where the heat flow is a contraction. The fact that the underlying constant $K=K(\Omega,w)$ in Szeg\"o's theorem can be taken to be the absolute constant 10 if $\Omega$ is a lens domain of the form \eqref{eq:Omega(r)def} and $w=1$ was shown in \cite[Proposition~15]{EI18}. This proves the Bernstein--Markov inequality \eqref{eq:gooddomain}.
\hfill$\Box$

\medskip

To derive the Bernstein--Markov inequalities for the discrete gradient presented in Theorem \ref{thm:gradR}, we will need to make use of Lust-Piquard's Riesz transform inequalities \cite{LP98} (see also \cite{BELP08} where the implicit dependence in $p$ was improved).

\begin{theorem} [Lust-Piquard]
For every $p\in[2,\infty)$, there exist $c_p,C_p\in(0,\infty)$ such that for every $n\in\N$, every function $f:\{-1,1\}^n\to\C$ satisfies
\begin{equation} \label{eq:lust}
c_p\|\Delta^{1/2}f\|_{L_p(\{-1,1\}^n;\C)} \leq \|\nabla f\|_{L_p(\{-1,1\}^n;\C)}\leq C_p \|\Delta^{1/2}f\|_{L_p(\{-1,1\}^n;\C)}.
\end{equation}
\end{theorem}

\medskip

\noindent {\it Proof of Theorem \ref{thm:gradR} for $p\geq2$.}
Fix $n\in\N$, $d\in\{1,\ldots,n\}$ and let $f:\{-1,1\}^n\to \C$ be a function of degree at most $d$. Then, by Lust-Piquard's inequality \eqref{eq:lust}, we have
\begin{equation} \label{eq:lust1}
\|\nabla f\|_{L_p(\{-1,1\}^n;\C)}\leq C_p \|\Delta^{1/2}f\|_{L_p(\{-1,1\}^n;\C)}.
\end{equation}
Combining \eqref{eq:lust1}, Naor and Schechtman's inequality \eqref{eq:naos} for $\beta=\frac{1}{2}$ and Theorem \ref{thm:laplacianR}, we derive the Bernstein--Markov inequalities of Theorem \ref{thm:gradR} for $p\geq2$.
\hfill$\Box$

\medskip

Even though the one-sided Riesz transform inequality 
\begin{equation} \label{eq:lust2}
\forall \ p\in(1,\infty), \ \ \ c_p\|\Delta^{1/2}f\|_{L_p(\{-1,1\}^n;\C)} \leq \|\nabla f\|_{L_p(\{-1,1\}^n;\C)}
\end{equation}
is true for $p\in(1,2)$ (see \cite{LP98}), its reverse is known to be false in this range. In fact, it has been shown by Naor and Schechtman (see \cite[Lemma~5.5]{BELP08}) that, if $p\in(1,2)$, a dimension independent inequality of the form
\begin{equation}
\|\nabla f\|_{L_p(\{-1,1\}^n;\C)}\leq C \|\Delta^{\beta}f\|_{L_p(\{-1,1\}^n;\C)}
\end{equation}
implies that $\beta\geq \frac{1}{p}$. We will now show that this is {\it almost} optimal. The following proposition is due to A.~Naor, to whom we are grateful for allowing us to include it here. The proof presented here is different than Naor's original proof, which will appear elsewhere.

\begin{proposition} [Naor] \label{prop:naor}
For every $p\in(1,2)$ and every $\e\in\big(0,\frac{1}{2}\big)$, there exists $C_p\in(0,\infty)$ such that for every $n\in\N$, every function $f:\{-1,1\}^n\to\C$ satisfies
\begin{equation} \label{eq:naor}
\|\nabla f\|_{L_p(\{-1,1\}^n;\C)}\leq \frac{C_p}{\e} \|\Delta^{1/p+\e}f\|_{L_p(\{-1,1\}^n;\C)}.
\end{equation}
\end{proposition}

For the proof of Proposition \ref{prop:naor} we will need the following lemma.

\begin{lemma} \label{lem:naor}
For every $p\in(1,2]$, there exists $A_p\in(0,\infty)$ such that for every $n\in\N$, every function $f:\{-1,1\}^n\to\C$ satisfies
\begin{equation} \label{eq:naor2}
\forall \ t>0, \ \ \ \| \nabla e^{-t\Delta} f\|_{L_p(\{-1,1\}^n;\C)} \leq \frac{A_p}{(e^{pt}-1)^{1/p}}\cdot \|f\|_{L_p(\{-1,1\}^n;\C)}.
\end{equation}
\end{lemma}

\begin{proof}
For $t\geq0$, consider the operator\mbox{ $S_t:L_p(\{-1,1\}^n;\C) \to L_p(\{-1,1\}^n\times\{-1,1\}^n;\C)$ given by}
\begin{equation}
S_t(f)(\e,\delta) \eqdef \sum_{A\subseteq\{1,\ldots,n\}} \widehat{f}(A) \prod_{i\in A} \big(e^{-t}\e_i + (1-e^{-t})\delta_i\big) - e^{-t\Delta} f(\e),
\end{equation}
where $(\e,\delta)\in\{-1,1\}^n\times\{-1,1\}^n$. Notice that, because of \eqref{eq:twovarcontract1}, \eqref{eq:twovarcontract2} and the contractivity of the heat semigroup, we have
\begin{equation} \label{200}
\|S_t(f)\|_{L_1(\{-1,1\}^n\times\{-1,1\}^n;\C)} \leq 2\|f\|_{L_1(\{-1,1\}^n;\C)}.
\end{equation}
Furthermore, if $f=w_A$ for a subset $A\subseteq\{1,\ldots,n\}$,
\begin{equation}
S_t(w_A)(\e,\delta) = \sum_{B\subsetneq A} e^{-t|B|} (1-e^{-t})^{|A\setminus B|}w_B(\e) w_{A\setminus B}(\delta),
\end{equation}
which by orthogonality implies that
\begin{equation}
\begin{split}
\|S_t(w_A)\|^2_{L_2(\{-1,1\}^n\times\{-1,1\}^n;\C)} & = \sum_{k=0}^{|A|-1} \binom{|A|}{k} e^{-2tk}(1-e^{-t})^{2(|A|-k)}
\\ & = \big(e^{-2t}+(1-e^{-t})^2\big)^{|A|} - e^{-2t|A|} \leq 1-e^{-t},
\end{split}
\end{equation}
where the last inequality is elementary. Therefore,
\begin{equation} \label{203}
\begin{split}
\|S_t\|_{{L_2(\{-1,1\}^n;\C)}\to L_2(\{-1,1\}^n\times\{-1,1\}^n;\C)} = \max_{A\subseteq\{1,\ldots,n\}} \|S_t(w_A)\|&_{L_2(\{-1,1\}^n\times\{-1,1\}^n;\C)}\\ & \leq \sqrt{1-e^{-t}}.
\end{split}
\end{equation}
Using the Riesz--Thorin interpolation theorem, we conclude that for every $p\in[1,2]$,
\begin{equation} \label{eq:gotdecay}
\|S_t\|_{{L_p(\{-1,1\}^n;\C)}\to L_p(\{-1,1\}^n\times\{-1,1\}^n;\C)} \stackrel{\eqref{200}\wedge\eqref{203}}{\leq} 2^{\frac{2}{p}-1} (1-e^{-t})^{1-\frac{1}{p}} \leq 2(1-e^{-t})^{1-\frac{1}{p}}
\end{equation}
and, by \eqref{eq:takeradproj}, we get
\begin{equation} \label{eq:almostthere!}
\begin{split}
\Big(\frac{1}{2^n}& \sum_{\delta\in\{-1,1\}^n} \Big\| \sum_{i=1}^n \delta_i  \partial_i e^{-t\Delta} f\Big\|_{L_p(\{-1,1\}^n;\C)}^p\Big)^{1/p} \stackrel{\eqref{eq:takeradproj}}{=} \frac{1}{e^t-1} \big\|\mathrm{Rad}_\delta S_t(f)\big\|_{L_p(\{-1,1\}^n\times\{-1,1\}^n;\C)}
\\& \leq \frac{K_p}{e^t-1} \|S_t(f)\|_{L_p(\{-1,1\}^n\times\{-1,1\}^n;\C)} \stackrel{\eqref{eq:gotdecay}}{\leq} \frac{2K_p\|f\|_{L_p(\{-1,1\}^n;\C)}}{e^{t(1-\frac{1}{p})}(e^t-1)^{\frac{1}{p}}} \leq \frac{4K_p\|f\|_{L_p(\{-1,1\}^n;\C)}}{(e^{pt}-1)^{\frac{1}{p}}},
\end{split}
\end{equation}
where $K_p=\sup_{n\in\N} \|\mathrm{Rad}\|_{L_p(\{-1,1\}^n;\C)\to L_p(\{-1,1\}^n;\C)}$. Finally, by Khintchine's inequality,
\begin{equation}
\| \nabla e^{-t\Delta} f\|_{L_p(\{-1,1\}^n;\C)} \leq B_p \Big(\frac{1}{2^n} \sum_{\delta\in\{-1,1\}^n} \Big\| \sum_{i=1}^n \delta_i  \partial_i e^{-t\Delta} f\Big\|_{L_p(\{-1,1\}^n;\C)}^p\Big)^{1/p}
\end{equation}
for some $B_p\in[1,\sqrt{2}]$, which combined with \eqref{eq:almostthere!} completes the proof.
\end{proof}

\smallskip

\noindent {\it Proof of Proposition \ref{prop:naor}.}
Fix $\e\in\big(0,\frac{1}{2}\big)$, $n\in\N$, $d\in\{1,\ldots,n\}$ and let $f:\{-1,1\}^n\to\C$. A change of variables gives the integral representation
\begin{equation} \label{eq:irep}
\Delta^{-1/p-\e} f = \frac{1}{\Gamma\big(\frac{1}{p}+\e\big)} \int_0^\infty t^{\frac{1}{p}+\e-1} e^{-t\Delta} f \diff t.
\end{equation}
Therefore, we can write
\begin{equation} \label{eq:userepre}
\begin{split}
\|\nabla \Delta^{-1/p-\e}f\|_{L_p(\{-1,1\}^n;\C)}& \stackrel{\eqref{eq:irep}}{=} \frac{1}{\Gamma\big(\frac{1}{p}+\e\big)} \Big\| \nabla \Big( \int_0^\infty t^{\frac{1}{p}+\e-1} e^{-t\Delta} f \diff t\Big) \Big\|_{L_p(\{-1,1\}^n;\C)}
\\ & \leq \frac{1}{\Gamma\big(\frac{1}{p}+\e\big)} \int_0^\infty t^{\frac{1}{p}+\e-1} \|\nabla e^{-t\Delta} f\|_{L_p(\{-1,1\}^n;\C)} \diff t
\\ & \stackrel{\eqref{eq:naor2}}{\leq} \frac{A_p}{\Gamma(1/2)} \Big( \int_0^\infty \frac{t^{\frac{1}{p}+\e-1}}{(e^{pt}-1)^{1/p}} \diff t \Big) \|f\|_{L_p(\{-1,1\}^n;\C)}.
\end{split}
\end{equation}
To conclude the proof, notice that
\begin{equation}
\int_0^\infty \frac{t^{\frac{1}{p}+\e-1}}{(e^{pt}-1)^{1/p}} \diff t  = \int_0^1 \frac{t^{\frac{1}{p}+\e-1}}{(e^{pt}-1)^{1/p}} \diff t  + \int_1^\infty \frac{t^{\frac{1}{p}+\e-1}}{(e^{pt}-1)^{1/p}} \diff t  = O(1/\e) + O(1)= O(1/\e)
\end{equation}
and thus \eqref{eq:userepre} becomes
\begin{equation}
\|\nabla \Delta^{-1/p-\e}f\|_{L_p(\{-1,1\}^n;\C)} \leq \frac{C_p}{\e}  \|f\|_{L_p(\{-1,1\}^n;\C)},
\end{equation}
which is equivalent to \eqref{eq:naor}.
\hfill$\Box$ 

\begin{remark}
Combining \eqref{eq:naor} with \eqref{eq:naos} and the estimate $\|\Delta f\|_{L_p(\{-1,1\}^n;\C)} \leq n\|f\|_{L_p(\{-1,1\}^n;\C)}$ which holds true for every function $f:\{-1,1\}^n\to\C$, we get 
\begin{equation}
\|\nabla f\|_{L_p(\{-1,1\}^n;\C)} \leq \frac{4C_p n^{\e}}{\e} \|\Delta^{1/p}f\|_{L_p(\{-1,1\}^n;\C)},
\end{equation}
which for $\e=\frac{1}{\log (n+1)}$ becomes
\begin{equation} \label{eq:almostriesz}
\|\nabla f\|_{L_p(\{-1,1\}^n;\C)} \leq 4eC_p \log (n+1) \|\Delta^{1/p}f\|_{L_p(\{-1,1\}^n;\C)}.
\end{equation}
In the upcoming manuscript \cite{EN20}, the first named author and Naor use a new Littlewood--Paley--Stein inequality \cite{Ste70} on the discrete hypercube, which allows to improve the logarithmic term in \eqref{eq:almostriesz} to $(\log n)^{c_p}$ for some $c_p\in(0,1)$, where $p\in(1,2)$. We conjecture that \eqref{eq:almostriesz} holds true with a dimension independent constant.
\end{remark}

\noindent {\it Proof of Theorem \ref{thm:gradR} for $p\in(1,2)$.}
Fix $\e\in(0,1)$, $n\in\N$, $d\in\{1,\ldots,n\}$ and let $f:\{-1,1\}^n\to \C$ be a function of degree at most $d$. Combining \eqref{eq:naor}, \eqref{eq:naos} for $\beta=\frac{1}{p}+\e$ and \eqref{eq:gooddomain}, we deduce that
\begin{equation}
\|\nabla f\|_{L_p(\{-1,1\}^n;\C)} \leq \frac{40 C_p}{\e} d^{(2-\frac{\theta_p}{\pi})(\frac{1}{p}+\e)} \|f\|_{L_p(\{-1,1\}^n;\C)}.
\end{equation}
Choosing $\e=\frac{1}{\log (d+1)}$ gives the conclusion.
\hfill$\Box$

\medskip

The next proposition is an asymptotically sharp endpoint Bernstein--Markov inequality for the discrete gradient and $p=\infty$.

\begin{proposition} \label{prop:gradp=infty}
For every $n,d\in\N$ with $d\in\{1,\ldots,n\}$ and every function $f:\{-1,1\}^n\to\C$,
\begin{equation} \label{eq:gradp=infty}
\|\nabla f\|_{L_\infty(\{-1,1\}^n;\C)} \leq 2d \|f\|_{L_\infty(\{-1,1\}^n;\C)}.
\end{equation}
Furthermore, the factor $d$ is asymptotically optimal.
\end{proposition}

We will need the following lemma, which is a special case of a more general semigroup statement from \cite[Proposition~8.6.1]{BGL14} and is the $L_\infty$ analogue of Lemma \ref{lem:naor}.

\begin{lemma} \label{lem:bgl}
For every $n\in\N$ and every function $f:\{-1,1\}^n\to\C$, we have
\begin{equation}
\forall \ t>0,  \ \ \ \|\nabla e^{-t\Delta} f\|_{L_\infty(\{-1,1\}^n;\C)} \leq \frac{1}{\sqrt{e^{2t}-1}} \|f\|_{L_\infty(\{-1,1\}^n;\C)}.
\end{equation}
\end{lemma}

\begin{proof}
Fix $n\in\N$ and consider a function $f:\{-1,1\}^n\to\C$. Then, we have the pointwise identity
\begin{equation} \label{eq:bglident}
(e^{2t}-1) \sum_{i=1}^n (\partial_i e^{-t\Delta}f)^2 = 2\int_0^t e^{2s} \sum_{i=1}^n (\partial_i e^{-t\Delta}f)^2 \diff s.
\end{equation}
Denote by $g_i(\e) = \e_i \partial_i e^{-(t-s)\Delta}f(\e)$ and notice that
\begin{equation}
\forall \ \e\in\{-1,1\}^n, \ \ \ \e_i e^s \partial_i e^{-t\Delta} f(\e) = e^{-s\Delta} g_i(\e).
\end{equation}
Therefore, \eqref{eq:bglident} implies that
\begin{equation}
\begin{split} \label{eq:bglineq}
(e^{2t}-1) \sum_{i=1}^n (\partial_i e^{-t\Delta}f)^2 & = 2\int_0^t  \sum_{i=1}^n (e^{-s\Delta}g_i)^2 \diff s \\ & \stackrel{(\dagger)}{\leq} 2\int_0^t e^{-s\Delta}\Big( \sum_{i=1}^n g_i^2 \Big) \diff s
 = 2\int_0^t e^{-s\Delta} \Big( \sum_{i=1}^n (\partial_i e^{-(t-s)\Delta}f)^2\Big) \diff s,
\end{split}
\end{equation}
where inequality $(\dagger)$ follows from Jensen's inequality since $e^{-s\Delta}$ is an averaging operator. Using the definition of the Laplacian, one can check that for every function $h:\{-1,1\}^n\to\C$ the identity
\begin{equation} \label{eq:byparts}
\forall \ \e\in\{-1,1\}^n, \ \ \ 2h(\e) \cdot \Delta h(\e) - \Delta h^2(\e) = 2\sum_{i=1}^n (\partial_i h(\e))^2
\end{equation}
holds true. Combining \eqref{eq:bglineq} and \eqref{eq:byparts}, we conclude that
\begin{equation} \label{eq:bglineq2}
(e^{2t}-1) \sum_{i=1}^n (\partial_i e^{-t\Delta}f)^2  \leq \int_0^t  e^{-s\Delta}\big( 2e^{-(t-s)\Delta}f \cdot \Delta e^{-(t-s)\Delta}f - \Delta (e^{-(t-s)\Delta}f)^2 \big) \diff s.
\end{equation}
However, we also have
\begin{equation}
\frac{\diff}{\diff s} e^{-s\Delta}\big( e^{-(t-s)\Delta} f \big)^2 = e^{-s\Delta} \big( 2e^{-(t-s)\Delta}f\cdot \Delta e^{-(t-s)\Delta}f - \Delta (e^{-(t-s)\Delta}f)^2 \big),
\end{equation}
which implies that \eqref{eq:bglineq2} can be rewritten as
\begin{equation} \label{eq:useinnextremark}
(e^{2t}-1) \sum_{i=1}^n (\partial_i e^{-t\Delta}f)^2  \leq \int_0^t \frac{\diff}{\diff s} e^{-s\Delta}\big( e^{-(t-s)\Delta} f\big)^2 \diff s = e^{-t\Delta} f^2 - (e^{-t\Delta} f)^2.
\end{equation}
Therefore,
\begin{equation}
\sqrt{e^{2t}-1} \|\nabla e^{-t\Delta} f\|_{L_\infty(\{-1,1\}^n;\C)} \leq \|e^{-t\Delta} f^2\|_{L_\infty(\{-1,1\}^n;\C)}^{1/2} \leq \|f\|_{L_\infty(\{-1,1\}^n;\C)},
\end{equation}
which completes the proof of the lemma.
\end{proof}

\begin{remark}
In fact, it follows from \eqref{eq:useinnextremark} that for every scalar valued function $f:\{-1,1\}^n\to\C$ and $p\in[2,\infty]$, we have the inequality
\begin{equation}
\forall \ t>0,  \ \ \ \|\nabla e^{-t\Delta} f\|_{L_p(\{-1,1\}^n;\C)} \leq \frac{1}{\sqrt{e^{2t}-1}} \|f\|_{L_p(\{-1,1\}^n;\C)}.
\end{equation}
Combined with Lemma \ref{lem:naor}, these estimates provide the sharp dimension free decay of the operator norm of $\nabla e^{-t\Delta}$ from $L_p(\{-1,1\}^n;\C)$ to itself for every $p\in(1,\infty]$.
\end{remark}

\medskip

\noindent {\it Proof of Proposition \ref{prop:gradp=infty}.}
Fix $n\in\N$, $d\in\{1,\ldots,n\}$ and let $f:\{-1,1\}^n\to \C$ be a function of degree at most $d$. Then, by Lemma \ref{lem:bgl}, Theorem \ref{thm:reverseheatgeneral} and the estimate $T_d(e^t)\leq e^{td^2}$, we have
\begin{equation}
\|\nabla f\|_{L_\infty(\{-1,1\}^n;\C)} \leq \frac{1}{\sqrt{e^{2t}-1}} \|e^{t\Delta}f\|_{L_\infty(\{-1,1\}^n;\C)} \leq \frac{e^{td^2}}{\sqrt{e^{2t}-1}} \|f\|_{L_\infty(\{-1,1\}^n;\C)}.
\end{equation}
Plugging $t=\frac{1}{d^2}$ and using that $\sqrt{e^{2t}-1} \geq \sqrt{2t}$, we thus get
\begin{equation}
\|\nabla f\|_{L_\infty(\{-1,1\}^n;\C)} \leq \frac{ed}{\sqrt{2}} \|f\|_{L_\infty(\{-1,1\}^n;\C)} < 2d\|f\|_{L_\infty(\{-1,1\}^n;\C)}.
\end{equation}
To prove that a $O(d)$ factor is necessary, consider the function $f:\{-1,1\}^n\to\C$ given by
\begin{equation}
\forall \ \e=(\e_1,\ldots,\e_n)\in\{-1,1\}^n, \ \ \ f(\e) \eqdef T_d\Big(\frac{\e_1+\cdots+\e_n}{n}\Big),
\end{equation}
where $T_d(x)$ is the $d$-th Chebyshev polynomial of the first kind. Clearly $f$ has degree at most $d$ and furthermore
\begin{equation}
\|f\|_{L_\infty(\{-1,1\}^n;\C)} = \max_{\e\in\{-1,1\}^n} \Big| T_d\Big(\frac{\e_1+\cdots+\e_n}{n}\Big)\Big| = T_d(1)=1.
\end{equation}
On the other hand,
\begin{equation}
\|\nabla f\|_{L_\infty(\{-1,1\}^n;\C)} \geq \Big( \sum_{i=1}^n (\partial_if(1,\ldots,1))^2\Big)^{1/2}
\end{equation}
and for every $i\in\{1,\ldots,n\}$,
\begin{equation}
\partial_if(1,\ldots,1) = \frac{1}{2}\Big(1-T_d\Big(1-\frac{2}{n}\Big)\Big).
\end{equation}
Therefore,
\begin{equation}
\|\nabla f\|_{L_\infty(\{-1,1\}^n;\C)} \geq \frac{\sqrt{n}}{2} \Big(1-T_d\Big(1-\frac{2}{n}\Big)\Big),
\end{equation}
which for $n=d^2$ becomes
\begin{equation} \label{eq:inft1}
\|\nabla f\|_{L_\infty(\{-1,1\}^n;\C)} \geq \frac{d}{2} \Big(1-T_d\Big(1-\frac{2}{d^2}\Big)\Big).
\end{equation}
Finally notice that
\begin{equation} \label{eq:inft2}
1-T_d\Big(1-\frac{2}{d^2}\Big) = 1 - \cos\Big(d\arccos\Big(1-\frac{2}{d^2}\Big)\Big) \geq 1
\end{equation}
for large enough values of $d$, since
\begin{equation}
\lim_{d\to\infty} 1 - \cos\Big(d\arccos\Big(1-\frac{2}{d^2}\Big)\Big) = 1-\cos(2).
\end{equation}
The asymptotic optimality of \eqref{eq:gradp=infty} follows from \eqref{eq:inft1} and \eqref{eq:inft2}.
\hfill$\Box$

\medskip

\begin{remark} \label{rem:gradp=1}
Using Theorem \ref{thm:momentcomp1}, one can also derive an endpoint Bernstein--Markov inequality for $p=1$. Indeed, using Jensen's inequality and orthogonality, for every $n,d\in\N$ with $d\in\{1,\ldots,n\}$, every function $f:\{-1,1\}^n\to\C$ satisfies
\begin{equation}
\|\nabla f\|_{L_1(\{-1,1\}^n;\C)} \leq \|\nabla f\|_{L_2(\{-1,1\}^n;\C)} \leq \sqrt{d} \|f\|_{L_2(\{-1,1\}^n;\C)}
\end{equation}
and invoking Theorem \ref{thm:momentcomp1} we conclude that
\begin{equation}
\|\nabla f\|_{L_1(\{-1,1\}^n;\C)} \leq (2.69076)^d\sqrt{d} \|f\|_{L_1(\{-1,1\}^n;\C)}.
\end{equation}
It would be interesting if one could obtain a polynomial bound in this inequality.
\end{remark}

\begin{remark}
In our recent paper \cite{EI18}, we investigated Bernstein--Markov inequalities in the spirit of Theorem \ref{thm:gradR} for polynomials on $\R^n$ equipped with the Gaussian measure. The main tools we used were Epperson's Gaussian complex hypercontractivity \cite{Epp89} (see also \cite{Jan97}) and Meyer's Gaussian Riesz transform inequalities \cite{Mey84}. Even though in \cite{EI18} we did not consider vector valued inequalities, we note that all the results of Section \ref{sec:3} also hold true in the Gaussian setting by the holomorphicity of the Ornstein--Uhlenbeck semigroup $e^{-tL}$ on $K$-convex targets which follows from Pisier's $K$-convexity theorem \cite{Pis82}. Moreover, it has been shown by Pisier \cite{Pis88} that for every UMD Banach space $(X,\|\cdot\|_X)$ and $p\in(1,\infty)$ there exist $c(p,X),C(p,X)\in(0,\infty)$ such that for every $n\in\N$ and every function $f:\R^n\to X$, we have
\begin{equation}
\begin{split}
c(p,X)\|L^{1/2}f\|_{L_p((\R^n,\gamma_n);X)} \leq \Big( \int_{\R^n} \Big\|\sum_{i=1}^n y_i\partial_i f\Big\|&_{L_p((\R^n,\gamma_n);X)} \diff\gamma_n(y)\Big)^{1/p}
\\ & \leq C(p,X) \|L^{1/2}f\|_{L_p((\R^n,\gamma_n);X)}.
\end{split}
\end{equation}
Therefore, since UMD spaces have non-trivial type (e.g. by \cite{Pis73}), we then conclude that for every $p\in(1,\infty)$ there exist $\beta=\beta(p,X)\in[0,1)$ and $K=K(p,X)\in(0,\infty)$ such that if $P$ is a polynomial on $\R^n$ of degree at most $d$, then
\begin{equation} \label{eq:freud}
\|\nabla P\|_{L_p((\R^n,\gamma_n);X)} \eqdef \Big( \int_{\R^n} \Big\|\sum_{i=1}^n y_i\partial_i f\Big\|_{L_p((\R^n,\gamma_n);X)} \diff\gamma_n(y)\Big)^{1/p} \leq K d^{\beta} \|P\|_{L_p((\R^n,\gamma_n);X)}.
\end{equation}
Such a result for UMD-valued functions on the Hamming cube is unknown. Inequality \eqref{eq:freud} is a high dimensional vector valued analogue of Freud's inequality \cite{Fre71}, who showed \eqref{eq:freud} for $X=\C$ and $n=1$ with the optimal exponent $\beta=\frac{1}{2}$ for every $p\in[1,\infty)$. To the best of our knowledge no vector valued versions of Freud's inequality (even for $n=1$) were available in the literature.
\end{remark}

We will now prove Corollary \ref{cor:influences}, the improved estimate on the influences of bounded functions. We will need the following lemma, whose proof relies on an inequality of Sarantopoulos \cite{Sar91}.

\begin{lemma} \label{lem:usesarantopoulos}
Let $n\in\N$, $d\in\{1,\ldots,n\}$ and $p\in[1,\infty)$. Then, for every function $f:\{-1,1\}^n\to\C$ of degree at most $d$, we have
\begin{equation}
\mathrm{Inf}^{(p)} f \leq \inf_{t>0} \left( \sup_{\mathrm{deg}(g)\leq d} \frac{\|g\|_{L_p(\{-1,1\}^n;\C)}^p}{\|e^{-t\Delta}g\|_{L_p(\{-1,1\}^n;\C)}^p} \right) \frac{d}{\sqrt{e^{2t}-1}} \|f\|_{L_\infty(\{-1,1\}^n;\C)}^p.
\end{equation}
\end{lemma}

\begin{proof}
For $t>0$, denote by
\begin{equation}
C_d(t) \eqdef \sup_{\mathrm{deg}(g)\leq d}  \frac{\|g\|_{L_p(\{-1,1\}^n;\C)}^p}{\|e^{-t\Delta}g\|_{L_p(\{-1,1\}^n;\C)}^p}
\end{equation}
and notice that
\begin{equation} \label{eq:useCd}
\mathrm{Inf}^{(p)} f  = \sum_{i=1}^n \|\partial_if\|^p_{L_p(\{-1,1\}^n;\C)} \leq C_d(t) \sum_{i=1}^n \|\partial_ie^{-t\Delta}f\|^p_{L_p(\{-1,1\}^n;\C)}.
\end{equation}
Moreover, we have
\begin{equation} \label{eq:boundCd}
\begin{split}
 \sum_{i=1}^n \|\partial_ie^{-t\Delta}f\|^p_{L_p(\{-1,1\}^n;\C)} \leq \Big(\sum_{i=1}^n \|\partial_ie^{-t\Delta}f\|_{L_1(\{-1,1\}^n;\C)}\Big) & \|e^{-t\Delta}f\|_{L_\infty(\{-1,1\}^n;\C)}^{p-1} 
\\ &\leq \|f\|_{L_\infty(\{-1,1\}^n;\C)}^{p-1} \mathrm{Inf}^{(1)}(e^{-t\Delta}f),
\end{split}
\end{equation}
where the last inequality follows from the contractivity of the heat semigroup. Therefore, it suffices to show that
\begin{equation} \label{eq:saranto2}
\mathrm{Inf}^{(1)} (e^{-t\Delta}f) \leq \frac{d}{\sqrt{e^{2t}-1}} \|f\|_{L_\infty(\{-1,1\}^n;\C)}.
\end{equation}
Let $F(x_1,\ldots,x_n)$ be the multilinear polynomial on $\R^n$ given by
\begin{equation}
\forall \ (x_1,\ldots,x_n)\in\R^n, \ \ \ F(x_1,\ldots,x_n) \eqdef \sum_{A\subseteq\{1,\ldots,n\}} \widehat{f}(A) \prod_{i\in A} x_i.
\end{equation}
Notice that for $\e\in\{-1,1\}^n$, we have $e^{-t\Delta}f(\e) = F(e^{-t}\e)$ and for $i\in\{1,\ldots,n\}$, $|\partial_i e^{-t\Delta} f(\e)| = e^{-t} |\partial_i F(\e)|$, where $\partial_i F$ is the usual partial derivative of $F$. Consequently,
\begin{equation} \label{eq:multilin}
\begin{split}
\mathrm{Inf}^{(1)}(e^{-t\Delta}f) = \frac{1}{2^n} & \sum_{\e\in\{-1,1\}^n} \sum_{i=1}^n |\partial_i e^{-t\Delta}f(\e)| \leq e^{-t} \max_{\e\in\{-1,1\}^n} \sum_{i=1}^n |\partial_i F(e^{-t}\e)|
\\& = e^{-t} \max_{\e,\delta\in\{-1,1\}^n} \sum_{i=1}^n \delta_i \partial_i F(e^{-t}\e) = e^{-t} \max_{\|x\|_\infty \leq e^{-t}, \|y\|_\infty \leq 1} \sum_{i=1}^n y_i \partial_i F(x),
\end{split}
\end{equation}
where in the last equality we used the multilinearity of the polynomial $\sum_{i=1}^n y_i \partial_i F(x)$ in the variables $(x,y)\in\R^{2n}$. By Sarantopoulos' vector valued Bernstein inequality from \cite{Sar91}, we deduce
\begin{equation} \label{eq:saranto}
\max_{\|x\|_\infty\leq e^{-t}, \|y\|_\infty \leq 1} \sum_{i=1}^n y_i \partial_i F(x) \leq \frac{d}{\sqrt{1-e^{-2t}}} \max_{\|x\|_\infty \leq 1} |F(x)| \leq \frac{d}{\sqrt{1-e^{-2t}}} \|f\|_{L_\infty(\{-1,1\}^n;\C)},
\end{equation}
again by the multilinearity of $F$. Therefore, \eqref{eq:multilin} and \eqref{eq:saranto} imply \eqref{eq:saranto2}, which combined with \eqref{eq:useCd} and \eqref{eq:boundCd} completes the proof of the lemma.
\end{proof}

Equipped with Lemma \ref{lem:usesarantopoulos}, we can prove Corollary \ref{cor:influences}.

\medskip

\noindent {\it Proof of Corollary \ref{cor:influences}.}
Fix $n\in\N$, $d\in\{1,\ldots,n\}$, $p\in\big(1,\frac{4}{3}\big)$ and let $f:\{-1,1\}^n\to \C$ be a function of degree at most $d$. Denote by
\begin{equation}
\eta_p \eqdef \frac{\pi}{2\pi - 2\arcsin\Big( \frac{2\sqrt{p-1}}{p}\Big)}.
\end{equation}
It follows from \eqref{eq:lowersmoothingR} and \eqref{eq:boundconformalbypower2} that there exists a constant $C_p\in(0,\infty)$ such that for every function $g:\{-1,1\}^n\to\C$ be a function of degree at most $d$,
\begin{equation}
\forall \ 0\leq t\leq1, \ \ \ \|e^{-t\Delta}g\|_{L_p(\{-1,1\}^n;\C)} \geq e^{-C_p t^{\eta_p}d}\|g\|_{L_p(\{-1,1\}^n;\C)}.
\end{equation}
Therefore, by Lemma \ref{lem:usesarantopoulos}, we conclude that
\begin{equation}
\mathrm{Inf}^{(p)} f \leq  \inf_{t>0} \frac{e^{pC_p t^{\eta_p}d}d}{\sqrt{e^{2t}-1}} \cdot \|f\|_{L_\infty(\{-1,1\}^n;\C)}^p.
\end{equation}
Choosing $t = d^{-1/\eta_p}$ and using that $\sqrt{e^{2t}-1}\geq \sqrt{2t}$, we deduce that there exists a constant $K_p\in(0,\infty)$ such that
\begin{equation} \label{eq:finalinfluencebound}
\mathrm{Inf}^{(p)} f \leq K_p d^{1+\frac{1}{2\eta_p}} \|f\|_{L_\infty(\{-1,1\}^n;\C)}^p = K_pd^{2-\frac{1}{\pi}\arcsin\big(\frac{2\sqrt{p-1}}{p}\big)} \|f\|_{L_\infty(\{-1,1\}^n;\C)}^p,
\end{equation}
which concludes the proof.
\hfill$\Box$

\medskip

Finally, we present the proof of Theorem \ref{thm:reversebernstein}.
\medskip

\noindent {\it Proof of Theorem \ref{thm:reversebernstein}.} Fix $n,d,m\in\N$ with $d+m\leq n$ and let $f:\{-1,1\}^n\to \C$ be a function of degree at most $d+m$ which is also in the $d$-th tail space. Then, by Lust-Piquard's lower Riesz transform inequalities \cite{LP98}, for every $p\in(1,\infty)$ there exists $c_p\in(0,\infty)$ such that
\begin{equation}
\|\nabla f\|_{L_p(\{-1,1\}^n;\C)} \stackrel{\eqref{eq:lust2}}{\geq} c_p \|\Delta^{1/2}f\|_{L_p(\{-1,1\}^n;\C)}.
\end{equation}
The conclusion of the theorem now follows from \eqref{eq:delta1/2}.
\hfill$\Box$

\bibliography{bernstein}
\bibliographystyle{alpha}
\nocite{*}

\end{document}